\RequirePackage{etoolbox}
\csdef{input@path}{{style/}{graphics/}}
\documentclass[aap,MSNbibl,seceqn,dvips]{arximspdf}
\usepackage{url,breakurl}

\doi{10.1214/14-AAP1008} 
\volume{25}
\issue{2}
\pubyear{2015}
\firstpage{675}
\lastpage{713}

\makeatletter

\newcommand{\rrvert}{\vert}
\newcommand{\llvert}{\vert}
\def\xrightarrow{\rightarrow}
\newcommand{\eqref}[1]{(\ref{#1})}
\newtheorem{theorem}{Theorem}[section]
\newtheorem{proposition}[theorem]{Proposition}
\newtheorem{lemma}[theorem]{Lemma}
\newtheorem{corollary}[theorem]{Corollary}

\newproclaim{definition}[theorem]{Definition}

\newproclaim{remark}[theorem]{Remark}

\newcommand{\del}{\partial}
\newcommand{\intE}{\int_{\rplus}}

\renewcommand{\epsilon}{\varepsilon}

\newcommand{\R}{\mathbb{R}}
\newcommand{\N}{\mathbb{N}}
\newcommand{\EE}{\mathbb{E}} 
\newcommand{\PP}{\mathbb{P}} 
\newcommand{\defeq}{\stackrel{ \operatorname{def}}{=}}
\newcommand{\laweq}{\stackrel{ \mathcal{L}}{=}}

\newcommand{\rplus}{E} 
\newcommand{\Mplus}{\mathcal{M}_{1 \wedge x}}

\newcommand{\qref}[1]{(\ref{#1})}

\newcommand{\vp}{\varphi}

\newcommand{\alphabr}{{\hat\alpha}}
\newcommand{\betabr}{{\hat\beta}}
\newcommand{\sspan}{\mathop{\operatorname{span}} }
\newcommand{\cadlag}{c\`{a}dl\`{a}g\ }
\newcommand{\Skor}{\mathbb{D}}

\makeatother

\begin{document}
\begin{frontmatter}

\title{Limit theorems for Smoluchowski dynamics associated with
critical continuous-state branching~processes}
\runtitle{Limit theorems for critical CSBPs}

\begin{aug}
\author[A]{\fnms{Gautam} \snm{Iyer}\thanksref{T1,T3}\ead[label=e1]{gautam@math.cmu.edu}},
\author[A]{\fnms{Nicholas} \snm{Leger}\corref{}\thanksref{T3}\ead[label=e2]{nleger@andrew.cmu.edu}}
\and
\author[A]{\fnms{Robert L.} \snm{Pego}\thanksref{T2,T3}\ead[label=e3]{rpego@cmu.edu}}
\runauthor{G.~Iyer, N.~Leger and R.~L. Pego}
\affiliation{Carnegie Mellon University}
\address[A]{Department of Mathematical Sciences\\
Carnegie Mellon University\\
Pittsburgh, Pennsylvania 15213\\
USA\\
\printead{e1}\\
\phantom{E-mail:\ }\printead*{e2}\\
\phantom{E-mail:\ }\printead*{e3}} 
\thankstext{T1}{Supported in part by NSF Grant DMS-10-07914.}
\thankstext{T2}{Supported in part by NSF Grants DMS-09-05723
and 1211161, and ICTI/FCT Grant UTA-CMU/0007/2009.}
\thankstext{T3}{Supported in part by the Center for Nonlinear Analysis
(CNA) under the NSF Grant no. 0635983 and PIRE Grant no. OISE-0967140.}
\end{aug}

\received{\smonth{1} \syear{2013}}
\revised{\smonth{9} \syear{2013}}

%
\begin{abstract}
We investigate the well-posedness and asymptotic self-similarity of
solutions to a generalized Smoluchowski coagulation equation recently
introduced by Bertoin and Le Gall
in the context of continuous-state branching theory.
In particular, this equation governs the
evolution of the L\'evy measure of a critical continuous-state
branching process which becomes extinct (i.e., is absorbed at zero)
almost surely.
We show that a nondegenerate scaling limit of the L\'evy measure
(and the process) exists
if and only if the branching mechanism is regularly varying at 0.
When the branching mechanism is regularly varying,
we characterize nondegenerate scaling limits
of arbitrary finite-measure solutions
in terms of generalized Mittag--Leffler series.
\end{abstract}

%

\end{frontmatter}

\section{Introduction}
\subsection{Overview}
Recently Bertoin and Le~Gall~\cite{BLG06} observed a connection
between the Smoluchowski coagulation equation
and any critical continuous-state branching process (hereafter CSBP) that
becomes extinct with probability one.
Our general goal in this paper is to establish criteria
for the existence of dynamic scaling limits in such branching processes,
by extending methods that were recently used to analyze coagulation
dynamics in the classically important ``solvable'' cases (i.e., cases
reduced to PDEs in terms of Laplace transforms).

Substantial progress has been made in recent years understanding
the long-time behavior of solutions to solvable Smoluchowski coagulation
equations. A~rich analogy has been developed between dynamic
scaling in these equations and classical limit theorems in probability,
including the central limit theorem, the classification of stable laws and
their domains of attraction \cite{Ley03,MP04}, and the L\'evy--Khintchine
representation of infinitely divisible laws \cite{Be02,MP08}.

A new challenge in dealing with
the coagulation equations that appear in the context of CSBPs
is that they typically lack the homogeneity properties which
were used extensively in earlier scaling analyses.
On the other hand, use of a Laplace exponent transform leads to the
study of a rather simple differential equation determined
by the branching mechanism of the CSBP. Moreover, these branching mechanisms
have a special structure, a L\'evy--Khintchine representation formula
expressed in terms of a certain measure related to family-size distribution.

To deal with the lack of homogeneity, we will adapt ideas from
re\-nor\-mal\-ization-group analysis, studying convergence
of rescaled solutions together with the rescaled equations they satisfy.
Such methods have been used to study asymptotic limits in a variety of
problems including nonlinear parabolic PDE and KAM theory
\cite{BK1}. 
An important point in this type of analysis, and one featured here, is
that nontrivial scaling limits, if they exist, satisfy a homogeneous
limiting equation. We describe these features in greater detail below.

\subsection{Continuous-state branching processes}\label{s:CSBP}
CSBPs arise as continuous-size,
continuous-time limits of scaled Galton--Watson processes,
which model the total number in a population of individuals
who independently reproduce with identical rates and family-size
distributions.
A CSBP consists of a two-parameter random process $(t,x)
\mapsto Z(t,x)\in[0,\infty)$ ($t\ge0$, $x>0$). For fixed $x$, the
process $t\mapsto Z(t,x)$ is Markov with initial value $Z(0,x) = x$.
For fixed $t$, the process $x \mapsto Z(t,x)$ is an increasing process
with independent and stationary increments. The right-continuous
version of this process
is a L\'evy process with increasing sample paths. In particular, the
process enjoys the branching property that $Z(t,x+y)$ has the same
distribution as
the sum of independent copies of $Z(t,x)$ and $Z(t,y)$ for all $t \geq0$.

The structure of the process $Z(t,x)$ has a precise characterization
via the Lamperti transform. That is, $t \mapsto Z(t,x)$ can be
expressed as a subordinated Markov process with parent process $x +
X_t$ where $X_t$ is a spectrally positive L\'evy process. More
specifically, $Z(t,x) = x + X_{\Theta(t,x)}$ where the process $t
\mapsto\Theta(t,x)$ has nondecreasing sample paths and formally
solves $\partial_t \Theta= x + X_\Theta$.
In this context, the Laplace exponent of $X_t$, denoted $\Psi$, is
called the \textit{branching mechanism} for $Z(t,x)$ and has L\'
evy--Khintchine representation
%
\begin{equation}
\label{e:bm1} \Psi(u) = \alpha u + \beta u^2 + \int
_{(0, \infty)} \bigl(e^{-ux} - 1 + ux\mathbf{1}_{\{x<1\}}
\bigr) \pi(dx),
\end{equation}
where $\alpha\in\R$, $\beta\geq0$, and $\int_{(0,\infty)} (1
\wedge
x^2) \pi(dx) < \infty$.
The representation~\eqref{e:bm1}, having the property
$\Psi(0^+) = 0$, assumes no killing for the associated
CSBP; cf.~\cite{Kyprianou}.

Due to the nature of the Lamperti transform, $Z(t,x)$ satisfies
%
\begin{equation}
\label{e:laplace1} \EE\bigl(e^{-q Z(t,x)}\bigr) = e^{-x\vp(t,q)},
\end{equation}
where the spatial Laplace exponent $\vp$ solves the backward equation
%
\begin{equation}
\label{e:back1} \del_t\vp(t,q) = -\Psi\bigl(\vp(t,q)\bigr),\qquad q\in(0,
\infty), t>0.
\end{equation}
Corresponding to $Z(0,x)=x$, the initial data takes the form $\vp(0,q)=q$.
It follows that $x \mapsto Z(t,x)$ is an increasing process with
independent and stationary increments. As the Laplace exponent of a
subordinator, $\varphi$ has the L\'evy--Khintchine representation
%
\begin{equation}
\label{e:khintchine} \vp(t,q) = b_t q + \int_{(0,\infty)}
\bigl(1-e^{-qx}\bigr)\nu_t(dx),\qquad q\ge0,
\end{equation}
where $b_t\ge0$ and $\int_{(0,\infty)} (1\wedge x)\nu_t(dx)<\infty$.
The quantities $b_t$ and $\nu_t$ represent the drift coefficient and
the L\'evy jump measure, respectively.
Taking $q \to\infty$ in~\eqref{e:laplace1} one sees that the CSBP
becomes extinct in time $t$ with positive probability (i.e., $\mathbb
{P}[Z(t,x) = 0] > 0$)
if and only if $\vp(t,\infty)<\infty$.
This means that $b_t=0$ and $\rho_t<\infty$, where
\[
\rho_t = \langle\nu_t, 1 \rangle\defeq\int
_{(0,\infty)}\nu_t(dx).
\]
(See Proposition~\ref{cmdphi} for a characterization of branching
mechanisms of this type.)

In the present work, we restrict our attention to the class of CSBPs
for which the branching mechanism $\Psi$ has the property
%
\begin{equation}
\label{e:conserve} \Psi^\prime\bigl(0^+\bigr) = \alpha- \int
_{[1,\infty)} x \pi(dx) > - \infty.
\end{equation}
That is, we assume $\Psi$ has the representation
%
\begin{equation}
\label{e:bm} \Psi(u) = \alphabr u + \betabr u^2 + \int
_{(0, \infty)} \bigl(e^{-ux} - 1 + ux\bigr) \pi(dx),
\end{equation}
where $\alphabr\in\R$, $\betabr\geq0$, and the \textit{branching
measure} $\pi(dx)$ verifies
%
\begin{equation}
\label{e:xminx2} \int_{(0,\infty)} \bigl(x \wedge x^2\bigr)
\pi(dx) < \infty.
\end{equation}
As shown in \cite{Grey,Kyprianou}, the CSBP associated
to~\eqref{e:bm}--\eqref{e:xminx2} is \textit{conservative} in the
sense that
$\PP(Z(t,x) < \infty) = 1$ for all $t >0$.
Of primary interest is the case of \textit{critical}
branching, which is distinguished by the property $\EE(Z(t,x))=x$, and
corresponds
here to the value $\alphabr= 0$.

\subsection{A generalized Smoluchowski coagulation equation}

The connection between branching and coagulation was described by
Bertoin and Le Gall
in~\cite{BLG06} as follows. Informally, the L\'evy measure $\nu_t(dx)$
corresponds to
the ``size distribution'' of the set of descendants of a single
individual at the initial time 0.
A more precise interpretation, when $b_t=0$,
is that $Z(t,x)$ is the sum of atoms of a Poisson measure on $(0,\infty
)$ with intensity $x\nu_t(dx)$.
Based on the study of the genealogy of CSBPs as in~\cite{DLG02} for example,
each of these atoms may be interpreted as the size of a clan of
individuals at time $t$ that
have the same ancestor at the initial time. [It is also possible to
interpret $\nu_t(dx)$ as
a continuum limit of scaled size distributions of clans descended from
a single ancestor
in a family of Galton--Watson processes. But precise discussion of this point
lies outside the present paper's scope, and is left for future work.]




As shown in \cite{BLG06}, the L\'evy measure of a critical CSBP
which becomes extinct almost surely satisfies a
generalized type of \textit{Smoluchowski coagulation equation}. This
equation belongs to a general class of coagulation models that account
for the simultaneous merging of $k$ clusters with (possibly time-dependent)
rate $R_k$.
Specifically, the weak form of this equation is
%
\begin{equation}
\label{e:gs} \frac{d\langle\nu_t, f \rangle}{dt} = \sum_{k\ge2}
R_k I_k(\nu_t,f)\qquad \mbox{for all } f \in C
\bigl([0, \infty]\bigr).
\end{equation}
Here
%
\begin{equation}
\label{d:Ik} I_k(\nu,f) = \int_{(0,\infty)^k}
\Biggl(f(x_1+\cdots+x_k)-\sum
_{i=1}^k f(x_i) \Biggr) \prod
_{i=1}^k \frac{\nu(dx_i)}{\langle\nu, 1 \rangle}
\end{equation}
represents the expected change in the moment
\[
\langle\nu,f \rangle \defeq\int_{(0,\infty)} f(x) \nu(dx)
\]
upon merger of $k$ clusters with size distribution $\nu$.
For the evolution equation of the L\'evy measure of a critical CSBP
which becomes extinct almost surely, the rate constants $R_k$ have a
particular Poissonian structure expressed
in terms of the branching mechanism and the total number
$\rho_t = \langle\nu_t, 1 \rangle$.
Namely $R_k=R_k(\rho_t)$ where 
%
\begin{equation}
\label{d:Rk} R_k(\rho) = \frac{(-\rho)^k \Psi^{(k)}(\rho)}{k!} = \int
_{(0,\infty)} \frac{
(\rho y)^k
}{k!} e^{-\rho y} \pi(dy) +
\delta_{k2} \betabr\rho^2.
\end{equation}
Here, $\betabr$ is the diffusion constant appearing in~\eqref{e:bm},
and $\delta_{k2}$ is the Kronecker delta function, which is zero for
$k\ge3$.
Combining the relations~\eqref{e:gs} and~\eqref{d:Rk} gives the
coagulation equation
%
\begin{equation}
\label{e:gsmain} \frac{d\langle\nu_t, f \rangle}{dt} = \sum_{k=2}^{\infty}
\frac
{(-\langle\nu_t, 1 \rangle)^k \Psi^{(k)}(\langle\nu_t, 1 \rangle)
}{k!} I_k(\nu_t,f).
\end{equation}

In the case of the special branching mechanism $\Psi(u) = u^2$, we
recover the classical Smoluchowski coagulation equation with rate
kernel $K(x,y) = 2$. Also, note that a L\'evy measure solution
of~\eqref
{e:gsmain} represents a kind of \textit{fundamental solution} for
the coagulation equation, having the special property that as $t\to0$
the measure $x\nu_t(dx)$ converges weakly to a delta function
at the origin; see Remark~\ref{massdelta}.

\subsection{Results and organization}

\subsubsection{Characterization of scaling limits for coagulation}
Our main results relate to long-time scaling limits
of measure solutions of the coagulation equation~\eqref{e:gsmain} where
$\Psi$ is a critical branching mechanism for a CSBP which becomes
extinct almost surely.
That is, we investigate the existence of dynamic scaling limits of the form
%
\begin{equation}
\label{l:scalelim} \alpha(t) \nu_{t} \bigl( \lambda(t)^{-1} \,dx
\bigr) \to\hat\nu(dx)\qquad \mbox{as } t \to\infty,
\end{equation}
for functions $\alpha, \lambda>0$ and a finite measure $\hat\nu$.
We show that the existence of nondegenerate limits is fundamentally
linked to two conditions:
\begin{longlist}[(ii)]
\item[(i)] \textit{regular variation of $\Psi$ at zero with index
$\gamma
\in(1,2]$};
\item[(ii)] \textit{regular variation of the mass distribution function
$\int_0^x y\nu_t(dy)$
at infinity with index $1-\rho$, where $\rho\in(0,1]$}.
\end{longlist}

First, assuming condition (i) holds, we prove (Theorem~\ref{main1}) that
scaling limits of form \eqref{l:scalelim} exist if and only if
condition (ii) holds at some initial time $t = t_0 \geq0$. Since
initial data satisfying (ii) are easily constructed, condition (i) gives
a sufficient condition under which~\eqref{e:gsmain} admits nontrivial
scaling solutions. The remarkable fact (Theorem~\ref{funthm}) is that
condition (i) is both \textit{necessary and sufficient} for the scaling
limit~\eqref{l:scalelim} to exist when $\nu_t$ is the \textit{fundamental
solution} (defined in Section~\ref{s:cmd}).

The theorems cited above also provide a precise characterization of the
limiting measure $\hat\nu$.
Specifically, we show that (i) and~\eqref{l:scalelim} together imply
that there exist constants $c_\lambda> 0$ and $\rho\in(0,1]$, the
latter given by (ii), such that
%
\begin{equation}
\label{mittagnu} \hat\nu(dx) = \langle\hat\nu, 1 \rangle F_{\gamma, \rho}\bigl(
\langle \hat \nu, 1 \rangle^{{1}/{\rho}} c_\lambda^{-1} \,dx
\bigr),
\end{equation}
where $F_{\gamma, \rho}$ is a generalized Mittag--Leffler probability
distribution given by
%
\begin{equation}
\label{gmlintro} F_{\gamma,\rho}(x) = \sum_{k=1}^{\infty}
\frac{(r)_k}{k!} \frac{(-1)^{k+1} x^{sk}}{\Gamma(sk + 1)},
\end{equation}
where
$r=(\gamma-1)^{-1}$, $s=\rho(\gamma-1)$, and
$(r)_k$ denotes the Pochhammer symbol
\[
(r)_k = r(r+1) (r+2)\cdots(r + k - 1).
\]
Moreover, the corresponding solution $\nu_t$ is \textit{asymptotically
self-similar} in the sense that for all $t>0$,
%
\begin{equation}
\label{l:scalelim2} \alpha(\tau) \nu_{\tau t} \bigl( \lambda(
\tau)^{-1} \,dx \bigr) \to t^{{1}/{(1-\gamma)}}\hat\nu\bigl(
t^{{1}/{(\rho(1-\gamma))}}\,dx\bigr)
\end{equation}
as $\tau\to\infty$.
In particular, the limiting function in~\eqref{l:scalelim2}
belongs to the family
of self-similar solutions
of~\eqref{e:gsmain} with homogeneous branching mechanism of the form
$\hat\Psi(u) = \beta u^\gamma$, where $\beta= (\gamma-1)^{-1}
\langle
\hat\nu, 1 \rangle^{1-\gamma}$.
These solutions have the form
\[
\nu_t(dx) = a(t) F_{\gamma, \rho}\bigl(a(t)^{{1}/{\rho}}
c_\lambda ^{-1} \,dx\bigr),\qquad  a(t) = \bigl[\beta(\gamma-1)t
\bigr]^{{1}/{(1-\gamma)}},
\]
which generalizes the one-parameter family obtained in \cite{MP04}
corresponding to the classical Smoluchowski equation, with $\gamma= 2$
and $c_\lambda= 1$.

\subsubsection{Limit theorems for critical CSBPs}

Theorems \ref{main1} and \ref{funthm} establish a necessary and
sufficient condition for the existence of nondegenerate scaling limits
of fundamental solutions, namely, condition (i), above. We now describe
two rather direct consequences of this fact in terms of scaling limits
of the corresponding CSBP.

First, given a CSBP $Z(t,x)$ for which the corresponding L\'{e}vy
measure is a fundamental solution of~\eqref{e:gsmain}, we consider
scaling limits of the form
%
\begin{equation}
\label{ssprocesslim} \lambda(t)Z\bigl(t, \alpha(t)x\bigr)
\mathop{\rightarrow}^{ \mathcal{L}}
\hat Z(x),
\end{equation}
with $\alpha(t)\to\infty$ and $\lambda(t)\to0$
as $t \to\infty$.
That is, we scale by a factor of $\lambda(t)$ the total population at
time $t$ descended from an initial population of size $\alpha(t)x$, %
and we investigate the convergence in law of the rescaled process, with
parameter $x$, to a nondegenerate L\'{e}vy process $\hat Z$.
As above, we prove that such a limit exists if and only if condition
(i) holds. This is Theorem~\ref{ssCSBP}. In particular, if~\eqref
{ssprocesslim} holds, then for each $t >0$,
%
\begin{equation}
\label{ssprocess3} \lambda(\tau)Z\bigl(\tau t, \alpha(\tau) x\bigr)
\mathop{\rightarrow}^ \mathcal{L}
t^{{1}/{(\gamma- 1)}} \hat Z\bigl(t^{{1}/{(1-\gamma)}}x\bigr)
\end{equation}
as~$\tau\to\infty$, and the right-hand side is equal in law to the
CSBP with L\'{e}vy measure given by $t^{{1}/{(1-\gamma)}}\hat\mu(
t^{{1}/{(1-\gamma)}}\,dx)$, where $\hat\mu$ is the L\'{e}vy measure of
$\hat Z$. In this way, we establish the self-similar form of the
limiting CSBP.

Alternatively, one can consider initial population as fixed, and obtain
a conditional limit theorem for critical continuous-state branching
processes conditioned on nonextinction.
In the context of discrete-state branching, several authors \cite
{Borovkov,Pakes,Slack72} have investigated limits of the form
%
\begin{equation}
\label{l:discrete} \PP\bigl(\lambda(t)Z_t \leq x | Z_t > 0
\bigr) \to F(x)
\end{equation}
as $t \to\infty$, where $Z_t$ is the branching process, and $F$ is a
nondegenerate distribution function on $(0,\infty)$. By various
techniques (our own being most similar to a method of Borovkov~\cite
{Borovkov}), the
authors prove that for the special scaling function $\lambda(t) = \PP(Z_t
> 0)$ a limit of the form~\eqref{l:discrete} exists if and only if the
process $Z_t$ has an offspring law corresponding to a regularly varying
probability generating function.
The question of whether the same regular variation condition is implied for
a general scaling function $\lambda(t) \to0$ was left open by
Pakes~\cite{Pakes}.
Theorem~\ref{limCSBP} provides an affirmative answer to the
continuous-state analog of the question posed by Pakes as an easy
corollary of Theorem~\ref{funthm}.

Also implied by Theorem~\ref{limCSBP} are the conditional limit
theorems obtained by Kyprianou \cite{Kyprianou2} for critical CSBPs
with power-law branching mechanism (the so-called $\alpha$-stable
case), and those obtained by Li \cite{Li00} for critical CSBPs with the
property $\Psi^{\prime\prime}(0^+) < \infty$. In all cases above,
including the discrete cases previously mentioned, limiting
distributions are characterized by relations of the form~\eqref{gmlintro}.

Let us note that, by comparison, noncritical CSBPs admit scaling
limits of a simpler form. Indeed, a well-known result of Grey~\cite
{Grey} states that for any supercritical CSBP with $\Psi^{\prime}(0^+)
> -\infty$, and for any critical or subcritical CSBP with $\Psi
^{\prime
}(\infty) < \infty$
[in the latter case, the CSBP remains positive almost surely---see
Proposition~\ref{cmdphi}(i)], there exists a scaling limit of the
form~\eqref{ssprocesslim}, where $\alpha(t) = 1$ and $\varphi(t,
\lambda
(t)) = \mathrm{const}$., with $\varphi$ solving~\eqref{e:back1}. On the other
hand, it follows directly from the work of Lambert~\cite{Lambert} that
any subcritical CSBP which becomes extinct almost surely admits a limit
of the form~\eqref{ssprocesslim} with scaling functions given by
$\alpha
(t) = 1/\varphi(t, \infty)$ and $\lambda(t) = 1$. In contrast with
Theorem~\ref{ssCSBP}, only one nontrivial scaling function is needed in
each of the cases above. 

\subsubsection{Well-posedness}

For the sake of completeness we also give an account of well-posedness
for the coagulation equation. That is, we establish the existence and
uniqueness of weak solutions of~\eqref{e:gsmain} when $\Psi$ is a
critical branching mechanism and the initial data is a finite measure
(Corollary~\ref{coreu}). Here, we essentially tie together the ideas of
Bertoin and Le Gall \cite{BLG06}, Norris \cite{Norris} and Menon and
Pego~\cite{MP04} with a few new proofs and observations. In particular,
we provide a simple and direct account of well-posedness for the
evolution of the L\'{e}vy measure $\nu_t$ in~\eqref{e:khintchine}
(Proposition~\ref{cmdphi}). The point is that equation \qref{e:back1}
preserves the property that $\vp(t,\cdot)$ has a completely monotone
derivative. For an initial cluster size distribution given by a finite
measure, the latter property amounts to a well-posedness result for
Smoluchowski dynamics.

\subsubsection{Outline of the paper}
We now give a brief outline of the paper. Section~\ref{s:prelim}
delineates some basic notation and definitions. Section~\ref{s:eu} is
dedicated to well-posedness results. In Section~\ref{s:gml},
we derive the family of self-similar solutions to~\eqref{e:gsmain}
associated with generalized Mittag--Leffler laws. Section~\ref{s:main1}
is dedicated to a study of scaling limits of the form~\eqref
{l:scalelim} in the case of a regularly varying branching mechanism
$\Psi$. In Section~\ref{s:main2}, we consider scaling limits of
fundamental solutions. Finally, in Section~\ref{s:limCSBP}, we reformulate
our scaling results
in terms of limit theorems for CSBPs.

\section{Preliminaries}\label{s:prelim}
We begin with some notation that will be repeatedly used throughout
this paper.
Let $\rplus$ be the open interval $(0,\infty)$, and $\overline
{\rplus}$
denote the extended interval $[0, \infty]$.
We use $C(\overline{\rplus})$ to denote the space of continuous
functions $f \dvtx\overline{\rplus} \to\R$, equipped with the
$L^\infty$-norm.

Three spaces of measures that arise often in our context are:
\begin{itemize}
\item
The space $\mathcal{M}_+$, consisting of positive Radon measures on
$\rplus$ equipped with the \emph{vague topology}.
We recall that if $\mu, \mu_1, \mu_2, \ldots$ are measures in
$\mathcal{M}_+$, then $\mu_n$ \textit{converges vaguely} to $\mu$
as $n
\to\infty$ (denoted by $\mu_n \mathop{\xrightarrow}\limits^{v} \mu$) if $\langle
\mu_n, \phi\rangle\to\langle\mu, \phi\rangle$ for all $\phi\in
C_c(\rplus)$.
Here $C_c(\rplus)$ denotes the space of continuous functions on
$\rplus
$ with compact support, and $\langle\mu, f \rangle$ denotes the
integral of $f$ with respect to the measure $\mu$.

\item
The space $\mathcal{M}_F$, consisting of finite positive measures on
$\rplus$, equipped with the \emph{weak topology}.
That is, if $\mu$, $\mu_1$, $\mu_2 ,\ldots$ are measures in
$\mathcal{M}_F$, then we say $\mu_n$ \textit{converges weakly} to
$\mu$ as
$n \to\infty$ (denoted by $\mu_n \mathop{\xrightarrow}\limits^{w} \mu$) if $\langle
\mu
_n, \phi\rangle\to\langle\mu, \phi\rangle$ for all $\phi\in
C_b(\rplus)$.
Here $C_b(\rplus)$ denotes the space of bounded continuous functions
on $\rplus$.

\item The space $\Mplus$, consisting of the set of measures $\mu\in
\mathcal{M}_+$ such that
\[
\int_{(0,\infty)} (1 \wedge x) \mu(dx) < \infty.
\]
\end{itemize}

\subsection{Branching mechanisms and Bernstein functions}

\begin{definition}
We say a function $\Psi\colon\rplus\to\R$ is a \emph{branching
mechanism} if it admits the representation
%
\begin{equation}
\label{eqnPsiRep} \Psi(u) = \alphabr u + \betabr u^2 + \intE
\bigl(e^{-ux} - 1 + ux\bigr) \pi(dx),
\end{equation}
where $\alphabr\in\R$, $\betabr\geq0$ and $\pi\in\mathcal{M}_+$
with $\intE(x \wedge
x^2) \pi(dx) < \infty$ (equivalently, $x \pi\in\Mplus$). The
branching mechanism is called \textit{critical, subcritical, or
supercritical} according to the conditions $\alphabr= 0$, $\alphabr>
0$, or $\alphabr< 0$, respectively.
\end{definition}

\begin{definition}\label{d:bernstein}
We say that $f \in C^\infty(\rplus)$ is a \textit{Bernstein
function} if
$f \geq0$ and $(-1)^k f^{(k+1)} \geq0$ for all integers $k \geq0$.
\end{definition}

In other words, $f$ is a Bernstein function if $f$ is nonnegative, and
$f^{\prime}$ is completely monotone.
It is well known (see, e.g.,~\cite{SSV}) that a function is
Bernstein if and only if it admits the representation
%
\begin{equation}
\label{eqnBernstein} f(q) = a + bq + \int_{\rplus}
\bigl(1-e^{-qx}\bigr) \mu(dx),
\end{equation}
where $a, b \geq0$ and $\mu\in\Mplus$.
Note that $f$ is strictly positive if and only if $(a, b, \mu) \ne
(0,0,0)$. On the other hand, a function $\Psi\colon\rplus\to\rplus$
belongs to the set of critical or subcritical branching mechanisms if
and only if $\Psi(0^+) = 0$ and $\Psi^{\prime}$ is a positive Bernstein
function.
The following lemma, for which we have found no obvious reference,
establishes a deeper relation between set of critical or subcritical
branching mechanisms and Bernstein functions.

\begin{lemma}\label{cmd1}
Assume $\Psi\colon\rplus\to\rplus$ is a critical or subcritical
branching mechanism. Then, the inverse function $\Psi^{-1}$ is a
Bernstein function.
\end{lemma}

\begin{pf} Let $f=\Psi^{-1}$ and $g=1/\Psi'$.
Note $g$ is completely monotone,
since $x\mapsto1/x$ is completely monotone, and
$\Psi'$ is a positive Bernstein function, as observed above.
Since $f$ is positive and $f'=g\circ f$,
it directly follows that $f'$ is completely monotone,
from \cite{MP08}, Lemma~5.5.
\end{pf}

\section{Well-posedness for Smoluchowski dynamics}\label{s:eu}

In this section we define a notion of weak solution for the generalized
Smoluchowski equation~\eqref{e:gsmain}.
As we will show, the question of existence of weak solutions
amounts to a study of~\eqref{e:back1}.
Several estimates appearing in Sections~\ref{s:weak} and \ref
{s:exist} have either been sketched in~\cite{BLG06} from a probabilistic point of view, or are straightforward
extensions of the well-posedness theory in \cite{MP04}.
The originality of our treatment lies mainly in Lemma~\ref{cmd1} and
its use in
the proof of Proposition~\ref{cmdphi}.
The remaining estimates have been simplified by various degrees and
organized for convenience of the reader.

\subsection{Weak solutions}\label{s:weak}

In this section, we consider a critical branching mechanism $\Psi$
having the representation~\eqref{eqnPsiRep} with $\alphabr= 0$.
Following the approach in \cite{MP04,Norris}, we associate to each
finite, positive measure $\nu\in\mathcal{M}_F$ the continuous linear
functional $L(\nu)\dvtx C(\overline{\rplus}) \to\mathbb{R}$,
defined by
%
\begin{eqnarray}
\label{eqnPsiWeak} \bigl\langle L(\nu), f \bigr\rangle&=& \sum
_{k\ge2} R_k\bigl(\langle\nu, 1 \rangle\bigr)
I_k(\nu,f) 
\nonumber
\\[-8pt]
\\[-8pt]
\nonumber
&= &\sum
_{k=2}^{\infty} \frac{(-\langle\nu, 1 \rangle)^k \Psi
^{(k)}(\langle\nu, 1 \rangle) }{k!} I_k(\nu,f),
\end{eqnarray}
where $I_k$ and $R_k$ are defined by~\eqref{d:Ik} and~\eqref{d:Rk},
respectively.
To verify continuity of $L(\nu)$, we observe
$|I_k(\nu,f)|\le(k+1)\|f\|_{C(\overline{\rplus})}$.
Thus for $m~=~\langle\nu, 1 \rangle$, equations~\eqref{eqnPsiRep}
and~\eqref{eqnPsiWeak} give
\begin{eqnarray*}
\bigl| \bigl\langle L(\nu), f \bigr\rangle\bigr| &=& \Biggl\llvert \betabr m^2
I_2(f) + \sum_{k=2}^{\infty}
{I_k(f)}\intE\frac{(mx)^k
}{k!}e^{-m x} \pi(dx)
\Biggr\rrvert
\\
& \leq&\| f \|_{C(\overline{\rplus})}
 \Biggl[3\betabr m^2 + \intE \Biggl[ \sum
_{k=2}^{\infty} \frac{m^k x^k}{(k-1)!} +\sum
_{k=2}^{\infty} \frac{m^k x^k}{k!} \Biggr]
e^{-mx} \pi(dx) \Biggr]
\nonumber
\\
&
=& K(m) \| f \|_{C(\overline{\rplus})},
\nonumber
\end{eqnarray*}
where $ K(m) = 3\betabr m^2 + 2m \Psi^{\prime}(m) - \Psi(m) < \infty$,
establishing continuity of $L(\nu)$.
Observe for future use, that
%
\begin{equation}
\label{cnstK} K(m) = \int_0^m \biggl[
\frac{3}{2}\betabr u+ 2u \Psi^{\prime\prime}(u) + \Psi^{\prime}(u)
\biggr] \,du.
\end{equation}
%
Since $\Psi^{\prime}, \Psi^{\prime\prime} \geq0$, the function
$m\mapsto K(m)$ is positive and increasing.

With this, the natural notion of weak solutions to~\eqref{e:gsmain} is
as follows.

\begin{definition}\label{weaksoln}
We say that a weakly measurable function $\nu\colon\rplus\to
\mathcal{M}_F$
is a \textit{weak solution} of~\eqref{e:gsmain} if
%
\begin{equation}
\label{weakform} \langle\nu_t, f \rangle= \langle\nu_s, f
\rangle+ \int_s^t \bigl\langle L(
\nu_{\tau}), f \bigr\rangle \,d \tau
\end{equation}
for all $t,s > 0$ and for all $f \in C(\overline{\rplus})$. If,
additionally, there exists $\nu_0 \in\mathcal{M}_F$ such that $\nu_t$
converges weakly to $\nu_0$ as $t \to0$, then we say $\nu\colon
[0,\infty) \to\mathcal{M}_F$ is a weak solution of~\eqref{e:gsmain}
with initial data $\nu_0$.
\end{definition}
%

To any function $\nu\colon\rplus\to\Mplus$, we associate the function
%
\begin{equation}
\label{defnu} \varphi(t,q) \defeq\int_{\rplus}
\bigl(1-e^{-qx}\bigr) \nu_t(dx).
\end{equation}
Our next result shows that weak solutions to~\eqref{e:gsmain} are
characterized by~\eqref{e:back1} for the associated function~$\varphi$.

\begin{theorem}\label{existgs1}
Let $\Psi\colon\rplus\to\rplus$ be a critical branching mechanism.
Assume $\nu\colon\rplus\to\mathcal{M}_F$ and that $\varphi$ is related
to $\nu$ by (\ref{defnu}).
Then, $\nu$ is a weak solution of~\eqref{e:gsmain}
if and only if $\varphi$ solves (\ref{e:back1}).
\end{theorem}

\begin{pf}
Let $\nu$ and $\varphi$ be as described.
First, we claim that
$\varphi$ satisfies~\eqref{e:back1} if and only if \eqref{weakform}
holds for the family of test functions
$f_q(x) \defeq1-e^{-qx}$, $0<q\leq\infty$.
Note, carefully, that we include the function
$f_{\infty} =1$ in this family. Indeed, since $\nu_t \in\mathcal
{M}_F$, we have
\[
\varphi(t,q) \to\varphi(t, \infty) \defeq\langle\nu_t, 1 \rangle<
\infty
\]
as $q\to\infty$, so that if~\eqref{e:back1} holds for $0 < q <
\infty$,
it also holds for $q = \infty$.

Note that since $\Psi$ is a critical branching mechanism given by
\eqref{eqnPsiRep}, it has an analytic extension defined in the right
half of
the complex plane.
Thus, the Taylor series of $\Psi(u)$ expanded about any $m>0$ converges
whenever $0<u<m$ and gives
%
\begin{equation}
\label{eqnPsiExpand}\qquad \Psi(u) = \sum_{k=0}^\infty
\frac{\Psi^{(k)}(m)}{k!} (u-m)^k,\qquad \Psi'(u) = \sum
_{k=0}^\infty\frac{\Psi^{(k)}(m)}{k!} (u-m)^{k-1}k.
\end{equation}
These formulas hold also for $u=0$, with $\Psi(0)=0$, $\Psi'(0)=0$,
due to the consistent sign of the terms for $k\ge2$.
Writing $m=\langle\nu_t,1 \rangle$, we compute that for $0 < q \leq
\infty$,
\begin{eqnarray*}
m^k I_k(\nu_t,f_q) &=& \int
_{\rplus^k} \Biggl[ f_q \Biggl(\sum
_{i=1}^k x_i \Biggr) - \sum
_{i=1}^k f_q(x_i)
\Biggr] \,d\nu_t^k
\\
&=& \int_{\rplus^k} \Biggl[ 1 - \prod_{i=1}^k
e^{-q x_i} -\sum_{i=1}^k \bigl(1 -
e^{-q x_i} \bigr) \Biggr] \,d\nu_t^k
\\
&= &m^k - \bigl(m-\varphi(t,q)
\bigr)^k - k m^{k-1}\varphi(t,q).
\end{eqnarray*}
Using this expression (which vanishes for $k=0$ and $1$)
in \eqref{eqnPsiWeak} and invoking \eqref{eqnPsiExpand}, since
$0<\varphi(t,q)<m$ we find
\begin{eqnarray*}
\bigl\langle L(\nu_t), f_q \bigr\rangle 
&=& \sum_{k=0}^\infty
\frac{(-1)^k \Psi^{(k)}( m)}{k!} \bigl[m^k - \bigl( m - \varphi(t, q)
\bigr)^k - k m^{k-1} \varphi(t, q) \bigr]
\\
&=& \Psi(0) - \Psi\bigl(\varphi(t, q)\bigr) + \varphi(t, q) \Psi'(0)
= - \Psi \circ\varphi(t, q).
\end{eqnarray*}
Therefore~\eqref{weakform} holds for $f = f_q$ if and only if
\[
\varphi(t,q) - \varphi(s,q) = -\int_s^t \Psi
\bigl(\varphi(\tau, q)\bigr) \,d \tau
\]
for all $s, t > 0$. This proves the claim.

In particular, if $\nu$ is a
weak solution of~\eqref{e:gsmain}, then $\varphi$ solves~\eqref{e:back1}.
On the other hand, if $\varphi$ solves~\eqref{e:back1},
then~\eqref{weakform} holds for all test functions $f_q$, $0 < q \leq
\infty$.
This family of test functions spans a dense subset of the metric space
$C(\overline{\rplus})$. Now, given $f \in C(\overline{\rplus})$ and
$\epsilon> 0$, choose $g \in
\sspan\{f_q\dvtx0 < q \leq\infty\}$ such that $\| f - g \| <
\epsilon$. By linearity, (\ref{weakform}) holds for the test
function~$g$. Therefore, assuming for definiteness that $t > s$, we have
\begin{eqnarray*}
&&\biggl\llvert \langle\nu_t, f \rangle- \langle\nu_s,f
\rangle- \int_s^t \bigl\langle L(
\nu_{\tau},f ) \bigr\rangle \,d \tau\biggr\rrvert
\\
&&\qquad = \biggl\llvert \langle\nu_t, f - g \rangle- \langle
\nu_s, f - g \rangle- \int_s^t \bigl
\langle L(\nu_{\tau}),f - g \bigr\rangle \,d \tau\biggr\rrvert
\\
&&\qquad \leq\|f-g\|_{C(\overline{\rplus})} \biggl( \langle\nu_t, 1
\rangle+ \langle\nu_s, 1 \rangle+ \int_s^t
K\bigl(\langle \nu_{\tau}, 1 \rangle\bigr) \,d \tau \biggr)
\\
&&\qquad \leq\epsilon\bigl[2\langle\nu_s, 1 \rangle+ (t-s) K\bigl(
\langle\nu_s, 1 \rangle\bigr)\bigr],
\end{eqnarray*}
where the function $K$ is given by (\ref{cnstK}), and we use
\eqref{e:back1} with $q=\infty$ to infer $\langle\nu_t,1 \rangle
\le\langle\nu_s,1 \rangle$.
Taking $\epsilon\to0$
shows that (\ref{weakform}) holds for all $f \in C(\overline{\rplus})$.
This completes the proof.
\end{pf}

\begin{remark}
Bertoin and Le Gall \cite{BLG06} propose a weaker form of
Smoluchowski's equation that requires only $\nu_t\in\Mplus$,
not $\mathcal{M}_F$, but which still transforms to~\eqref{e:back1}. In
particular, they show that if $\Psi^{\prime}(\infty) = \infty$ (see
Proposition~\ref{cmdphi}, below), then the
L\'evy measure of the associated CSBP verifies this weak form for the
special test functions $f_q$,
$0<q<\infty$. However, there appear to be no obvious estimates
available to deal with a general test function $f \in C(\overline
{\rplus})$.
%
\end{remark}

\subsection{Fundamental solutions}\label{s:cmd}
For any weak solution $\nu\colon\rplus\to\mathcal{M}_F$ of the
generalized Smoluchowski equation \eqref{e:gsmain}, the solution
$\varphi(t,q)$ of \eqref{e:back1} has a finite limit as $t \to0$
whether or not $\nu_t$ has a weak limit as $t \to0$. Indeed, if
$\varphi(t,q_0) \to\infty$ as $t \to0$ for some $q_0>0$, then, by a
translation invariance of solutions, one shows that for any $q>q_0$
there exists $t_q >0$ such that
$\varphi(t,q) = \varphi(t-t_q, q_0) \to\infty$ as $t \to t_q$, which
contradicts $\nu_{t_q} \in\mathcal{M}_F$.

It follows that $\varphi$ has the convenient representation
%
\begin{equation}
\label{fsoln} \varphi(t,q) = \Phi\bigl(t, \varphi(0,q)\bigr),
\end{equation}
where $\varphi(0,q) \defeq\varphi(0^+,q)$ and where $\Phi$ solves the
initial value problem
%
\begin{equation}
\label{IVPphi}\cases{ %
\partial_t
\Phi(t,q) = -\Psi\bigl(\Phi(t,q)\bigr),&  \quad $q \in\rplus,$
\vspace*{2pt}\cr
\Phi(0,q) = q. }
\end{equation}
The functions $\Phi_t=\Phi(t,\cdot)$ have the semigroup property
$\Phi_{t+s}=\Phi_t\circ\Phi_s$ for \mbox{$t, s>0$}.
Because of the composition structure~\eqref{fsoln}, we make the
following definition.

\begin{definition}\label{funsoln}
Assume $\Psi\colon\rplus\to\rplus$ is a critical branching
mechanism. We say that a function $\mu\colon\rplus\to\mathcal{M}_F$
is the \textit{fundamental solution} of the generalized Smoluchowski
equation~\eqref{e:gsmain} if the function
%
\begin{equation}
\label{funphi} \Phi(t,q) = \intE\bigl(1-e^{-qx}\bigr)
\mu_t(dx)
\end{equation}
solves the initial value problem~\eqref{IVPphi}, where $\Phi(0,q)
\defeq\Phi(0^+,q)$.
\end{definition}


The fundamental solution relates solutions of the generalized
Smoluchow\-ski
equation to their initial data via solutions of a linear problem; see
Remark~\ref{subord}
below for details. But first we establish necessary and sufficient
criteria for the existence of a fundamental solution,
and develop the basis for our discussion of
well-posedness theory for weak solutions with initial data.

\begin{definition}
We say that a branching mechanism $\Psi\dvtx\rplus\to\R$
satisfies \emph
{Grey's condition}~\cite{Grey} provided $\Psi(\infty) = \infty$ and
%
\begin{equation}
\label{conditionE} \int_a^{\infty} \frac{1}{\Psi(u)}
\,du < \infty \qquad\mbox{for some $a>0$.}
\end{equation}
\end{definition}

\begin{remark}\label{rmkE}
It is well known that Grey's condition gives a necessary and sufficient
condition under which solutions to \eqref{e:back1} have finite-time
blow-up, backward in time.
We also mention that Bertoin and Le Gall~\cite{BLG06} use the term
\emph
{Condition~E} to describe Grey's condition.
\end{remark}

\begin{proposition}\label{cmdphi}
Let $\Phi$ be the unique solution of the initial value problem~\eqref
{IVPphi}, where $\Psi\colon\rplus\to\R$ is any branching mechanism
of the form~\eqref{eqnPsiRep}.
Then, for each fixed $t \geq0$, the map $\Phi(t, \cdot) \dvtx\rplus
\to
\rplus$ is a Bernstein function. More precisely,
%
\begin{equation}
\label{eqnPhiRep} \Phi(t,q) = b_t q + \intE\bigl(1-e^{-qx}
\bigr) \mu_t(dx)
\end{equation}
for some $b_t \geq0$ and $\mu_t \in\Mplus$. Furthermore, the
following properties hold:
\begin{longlist}[(ii)]
\item[(i)] $b_t = 0$ for some (equivalently all) $t >0$ if and only if
$\Psi^{\prime}(\infty) = \infty$;
\item[(ii)] $b_t = 0$ and $\mu_t \in\mathcal{M}_F$ for some
(equivalently all) $t>0$ if and only if $\Psi$ satisfies Grey's condition.
\end{longlist}
\end{proposition}
%

\begin{remark}
While the facts above can be infered from CSBP theory, we summarize
them here for convenience of the reader, and give a proof independent
of the latter theory. In particular, we recognize equation~\eqref
{eqnPhiRep} as the L\'evy--Khintchine formula for the Laplace exponent
of a CSBP with branching mechanism $\Psi$, as sketched in Section~\ref
{s:CSBP}. In this context, property (ii) states that a CSBP becomes
extinct by time $t$ with positive probability ($\Phi(t, \infty) <
\infty
$) if and only if Grey's condition holds. For critical CSBPs, this is
the case if and only if the process becomes extinct almost surely.
Thus property (ii) establishes a one-to-one correspondence between
fundamental solutions of \eqref{e:gsmain} and
L\'{e}vy measures for critical CSBPs that become extinct almost surely.
\end{remark}

\begin{pf}
Our proof is based on the implicit Euler method. First we will show that
each iteration of the implicit Euler scheme for~\eqref{IVPphi} yields a
Bernstein function. Then, since the set of Bernstein functions is closed
under composition and pointwise limits \cite{SSV}, pages~20--21,
convergence of the implicit Euler scheme implies that $\Phi(t, \cdot)$
is Bernstein.

By assumption, $\Psi$ has the representation~\eqref{eqnPsiRep}. Since
$\Psi^\prime$ is increasing and $\Psi^\prime(0^+) = \alphabr\in\R
$, it
follows that $\Psi$ is Lipshitz on bounded intervals.
Hence~\eqref{IVPphi} has a unique solution. Furthermore, the solution
remains positive for all time since the equation is autonomous and
$\Psi(0^+) = 0$. Also, since $\partial_t \Phi= -\Psi(\Phi) \leq
-\alphabr\Phi$,
we obtain, for all $t, q \geq0$, the bound
%
\begin{equation}
\label{phibound} 0 \leq\Phi(t,q) \leq qe^{-\alphabr t}.
\end{equation}

For fixed $t > 0$ and $N \in\N$, let $h=t/N$ and consider the
iteration scheme
%
\begin{equation}
\label{scheme1} \hat\Phi_{n+1}(q) = \hat\Phi_n(q) - h \Psi
\bigl( \hat\Phi_{n+1} (q)\bigr),\qquad n=0,1,\ldots,N-1.
\end{equation}
%
Note that for $N$ sufficiently large, the function $F_N\dvtx\rplus\to
\rplus$ defined by
%
\begin{equation}
\label{eulerF} F_N(u) = u + \frac{t}{N}\Psi(u) = u + h
\Psi(u)
\end{equation}
is a bijection, since
$F_N^{\prime}(u) = 1 + h\Psi^\prime(u) \geq1+\alphabr h>0$.
By consequence, $\hat\Phi_{n+1}(q) = F_N^{-1} (\hat\Phi_{n}(q))$
is well-defined and positive for all $q>0$ and $n=0,1,\ldots,N-1$.
Since $\Psi$ is locally smooth on $\rplus$ and we have the bound
\eqref{phibound}, the proof of the pointwise
convergence $\hat\Phi_N(q)\to\Phi(t,q)$ as $N\to\infty$
for each $q>0$ is standard, and we omit it.

Observe now that $F_N$ is a branching mechanism since it has a
representation of the form \eqref{eqnPsiRep}. Hence by Lemma~\ref{cmd1},
$F_N^{-1}$ is a Bernstein function, provided $N$ is sufficiently large.
Since the set of Bernstein functions is closed under composition,
and $\hat\Phi_0(q)=q$ is a Bernstein function, it follows
$\hat\Phi_n$ is a Bernstein function for each $n=0,\ldots,N$.
Finally, the pointwise convergence $\hat\Phi_N \to\Phi(t,\cdot)$
as $N
\to\infty$ implies that $\Phi(t,\cdot)$ is a Bernstein function, by
\cite{SSV}, Corollary~3.8.
Representation~\eqref{eqnBernstein} then gives
%
\begin{equation}
\label{cmd2} \Phi(t,q) = a_t + b_t q + \int
_{\rplus} \bigl(1-e^{-qx}\bigr) \mu_t(dx),
\end{equation}
for some $a_t, b_t \geq0$ and $\mu_t \in\Mplus$.
Note that (\ref{phibound}) implies $a_t = \Phi(t,0^+) = 0$ for all
$t\geq0$.

Next we establish (i). Observe that
$b_t = \partial_q \Phi(t,\infty)$,
and that the relation
%
\begin{equation}
\label{ode2} \partial_q \Phi(t,q) = e^{-\int_0^t \Psi^{\prime}(\Phi(s,q)) \,ds}
\end{equation}
is an easy consequence of (\ref{e:back1}).
If $\Psi^{\prime}(\infty) < \infty$, then since $\Psi'$ and $\Phi
(s,\cdot)$ are increasing,
for any $t>0$ we find
\[
b_t = e^{-\int_0^t \Psi^{\prime}(\Phi(s,\infty)) \,ds} \geq e^{-t\Psi
^{\prime}(\infty)} > 0.
\]
Conversely,
suppose $b_t > 0$ for some $t > 0$, then (\ref{cmd2}) implies $\Phi(t,
\infty) = \infty$ and hence $\Phi_t=\Phi(t,\cdot)$ is a surjection onto
$\rplus$.
Since $\Phi_t=\Phi_s\circ\Phi_{t-s}$ for $0<s<t$, $\Phi_s$ is also
a surjection
and hence $\Phi(s,\infty)=\infty$. Thus
$b_t = e^{-t\Psi^{\prime}(\infty)}>0$.
Hence, $\Psi^{\prime}(\infty) < \infty$.
This completes the proof of (i).

Finally, let us show that (ii) holds. First suppose $b_t = 0$ and $\mu
_t \in\mathcal{M}_F$ for some $t>0$. We claim Grey's condition holds.
From \eqref{eqnPhiRep}
we have
%
\begin{equation}
\label{totalmu} \Phi(t,\infty) = \intE\mu_t(dx) < \infty.
\end{equation}
Assume for the sake of contradiction that $\Psi(\infty) < \infty$. Then
$\Psi^\prime(\infty) \leq0$, and we have by~\eqref{eqnPsiRep},
$\betabr=
0$ and $\alphabr= \Psi^{\prime}(0^+) \leq-\int x \pi(dx)$.
In that case, $\Psi(u) < \int_{\rplus} (e^{-ux} - 1) \pi(dx) < 0$ for
all $u \in\rplus$, and $\Phi(\cdot, q)$ is increasing. Hence $\Phi
(t,q) \geq q \to\infty$ as $q \to\infty$, which contradicts~\eqref
{totalmu}. This shows $\Psi(\infty) = \infty$.

Now, assume~\eqref{conditionE} fails. As remarked above, failure of
this condition ensures that all solutions of (\ref{e:back1}) with
finite initial data remain finite backward in time. In particular, by
uniqueness and positivity of solutions of (\ref{IVPphi}), we have that
for all $q>0$,
\[
q=\Phi(0,q) = \Phi\bigl(-t, \Phi(t,q)\bigr) \leq\Phi\bigl(-t, \Phi(t,\infty)
\bigr),
\]
which is finite and independent of $q$.
Note that we used monotonicity of $\Phi$ in $q$ for the inequality.
This is a contradiction. Hence, Grey's condition holds.

Conversely, assume Grey's condition holds, and let
\[
q_* \defeq\inf\bigl\{q\in\rplus\dvtx\Psi(q) >0\bigr\}
\]
denote the largest equilibrium solution of~\eqref{e:back1}. Then, for
any $q>q_*$
there exists $t_q <0$ such that $\rplus\owns\Phi(t, q) \to\infty$ as
$t \to t_q^+$.
We define the special solution
\[
\Phi_*(t) = \Phi(t-t_q,q),
\]
which is independent of $q>q_*$ and has the property $\rplus\owns\Phi
_*(t) \to\infty$ as $t \to0^+$.
Since $\Phi(0,q) < \Phi_*(0^+) = \infty$, we deduce, by uniqueness of
solutions of~\eqref{e:back1}, that $\Phi(t,q) < \Phi_*(t)$ for all $t,
q > 0$. Therefore, taking $q\to\infty$, shows
$\Phi(t, \infty) < \infty$. That is, $b_t =0$ and $\mu_t \in
\mathcal{M}_F$,
for all $t>0$.
\end{pf}

\begin{remark}\label{massdelta}
Note that the Bernstein functions $\Phi(t, \cdot)$ converge pointwise
to the function $\Phi(0,q) = q$ as $t \to0$. It follows that
\[
\partial_q \Phi(t, q) = b_t + \int
_{\rplus} e^{-qx} x\mu_t(dx)
\]
converges pointwise to $\partial_q \Phi(0,q) = 1 = \int_{[0,\infty)}
e^{-qx} \delta_0(dx)$ as $t \to0$; see, for instance,~\cite{SSV}, page~21.
Therefore, by the continuity theorem (cf. \cite{Feller}, Theorem~XIII.1.2), the
measures $\kappa^{(t)}(dx) = b_t \delta_0(dx) + x\mu_t(dx)$
converge vaguely to the measure $\delta_0(dx)$ in the space of positive
Radon measures on $[0,\infty)$. In particular, if $\Psi$ is a critical
branching mechanism satisfying Grey's condition, then the mass measure,
$x\mu_t(dx)$, converges weakly to a delta mass at zero as $t \to0$.
Moreover, the total mass at time $t$, given by $\partial_q \Phi(t,
0^+)$, is conserved by~\eqref{ode2}.
\end{remark}

\begin{remark}\label{subord} Formula \eqref{fsoln} has a standard
probabilistic
interpretation: For fixed $t$, the L\'evy process with L\'evy measure
$\nu_t$ is subordinated to
the L\'evy process with L\'evy measure $\nu_0$ by the directing process
$Z(t,\cdot)$ with L\'evy measure $\mu_t$.
In terms of generators, this corresponds, however, to a deterministic
formula [\eqref{e:subordinate} below] that expresses the weak
solution $\nu_t$ of the nonlinear
Smoluchowski equation in terms of the fundamental solution $\mu_t$
and the kernel $Q_s$ of a convolution semigroup (a L\'evy diffusion)
given by
%
\begin{equation}
\label{e:gen1} e^{sA}f(x)=\int_{\R}
f(x+y)Q_s(dy),
\end{equation}
with generator $A$ determined from $\nu_0(dx)$ by
%
\begin{equation}
\label{e:gen2} Af(x) = \int_{\rplus} \bigl(f(x+y)-f(x)\bigr)
\nu_0(dy),
\end{equation}
for all smooth $f\in C_c(\R)$. Supposing that
$f(x)=e_q(x):=e^{-qx}$ for $x\ge0$,
we find that for $x\ge0$,
%
\begin{equation}
\label{e:gen3}\qquad Af(x) = -\varphi(0,q) e_q(x),\qquad e^{sA}f(x)
= e_q(x) \int_{[0,\infty)} e^{-qy}Q_s(dy).
\end{equation}
[Note $Q_s(dx)$ retains an atom at 0 with mass
$e^{-s\langle\nu_0,1 \rangle}$.] Hence
\[
\int_{E}\bigl(1-e^{-qx}\bigr)\int
_E Q_s(dx)\mu_t(ds) = \int
_E \bigl(1-e^{-s\varphi(0,q)}\bigr) \mu_t(ds) =
\Phi\bigl(t,\varphi(0,q)\bigr).
\]
Consequently, from \eqref{fsoln} we infer that
%
\begin{equation}
\label{e:subordinate} \nu_t(dx)= \int_{\rplus}
Q_s(dx) \mu_t(ds).
\end{equation}
Note that $Q_s$ is determined by solving
a linear equation, namely
$\partial_t u = Au$.
\end{remark}

\subsection{Weak solutions with initial data}\label{s:exist}

In this section we establish the existence and uniqueness of weak
solutions of~\eqref{e:gsmain} with initial data $\nu_0 \in\mathcal{M}_F$.

\begin{lemma}\label{nulemma}
Assume $\Psi\colon\rplus\to\R$ is a branching mechanism, and let
$\nu
_0 \in\Mplus$.
Then, there exists a unique vaguely continuous map $\nu\colon[0,
\infty
) \to\mathcal{M}_+$ such that $\varphi$, defined by~\eqref{defnu}, is
a solution of (\ref{e:back1}) with initial data
%
\begin{equation}
\label{dataphi0} \varphi_0(q) = \int_{\rplus}
\bigl(1-e^{-qx}\bigr) \nu_0(dx).
\end{equation}
Furthermore, if $\nu_0 \in\mathcal{M}_F$, then $\nu_t \in
\mathcal{M}_F$ for all $t >0$, and $\nu\colon[0, \infty) \to
\mathcal
{M}_F$ is weakly continuous. In this case, we have for~$t>0$
%
\begin{equation}
\label{totalnumber} \frac{d \langle\nu_t, 1 \rangle}{dt} = -\Psi\bigl( \langle\nu_t, 1
\rangle\bigr).
\end{equation}
\end{lemma}
%

\begin{pf}
Note that~\eqref{fsoln} represents the unique solution of (\ref
{e:back1}) with initial data $\varphi(0,q) = \varphi_0(q)$. By
Proposition~\ref{cmdphi}, the map $\Phi(t, \cdot)$ is a Bernstein
function for all $t \geq0$. Also, $\varphi_0$ is a Bernstein function
as it admits a representation of the form~\eqref{eqnBernstein}.
Therefore the composite function, given by (\ref{fsoln}), is a
Bernstein function for all $t \geq0$. Furthermore, we have $\varphi(t,
0^+) = 0$ and
\[
\partial_q\varphi(t, \infty) = \lim_{q \to\infty}
\partial_q\Phi\bigl(t, \varphi_0(q)\bigr)\cdot\lim
_{q \to\infty} \varphi_0^{\prime}(q).
\]
By assumption, the latter limit vanishes, and since $\partial_q\Phi(t,
\cdot)$ is decreasing, the former limit is finite. Hence, $\partial
_q\varphi(t, \infty) = 0$ and representation~\eqref{eqnBernstein} for
$\varphi$ reduces to
%
\begin{equation}
\label{varphi1} \varphi(t,q) = \int_{\rplus} \frac{1-e^{-qx}}{x}
\mu_t(dx),
\end{equation}
for some $\mu_t \in\mathcal{M}_+$ with $x^{-1} \mu_t \in\Mplus$.
Note that the measure $\mu_t$ is determined uniquely by its Laplace
transform $\partial_q \varphi(t,q)$.
Further, $\partial_q \varphi(t,q)$ is continuous in~$t$, since
$\partial
_t\, \partial_q \varphi(t,q) = \Psi^{\prime}(\varphi(t,q)) \Psi
(\varphi
(t,q))$. Therefore, viewing $\mu_t$ as a measure on $[0,\infty)$ (which
assigns measure zero to the point $\{0\}$), it follows from the
continuity theorem (cf. \cite{Feller}, Theorem~XIII.1.2) that the map $\mu
\colon
[0,\infty) \to\mathcal{M}_+([0,\infty))$ is vaguely continuous, where
$\mathcal{M}_+([0,\infty))$ is the space of Radon measures on
$[0,\infty)$.
In particular, for all $f \in C_c(\rplus) \subset C_c([0,\infty))$, we
have $\langle f, \mu_{s} \rangle\to\langle f, \mu_{t}\rangle$ as $s
\to t$. That is, the map $\mu\dvtx[0,\infty) \to\mathcal{M}_+$ is
vaguely continuous. Hence, the map
$ t \mapsto\nu_t \defeq x^{-1} \mu_t \in\mathcal{M}_+$ is also
vaguely continuous [since $\mu_{s} \mathop{\xrightarrow}\limits^{v} \mu_t$ implies
$\phi
\cdot\mu_{s} \mathop{\xrightarrow}\limits^{v} \phi\cdot\mu_t$ for any $\phi\in
C_c(\rplus)$]. This establishes the first part of the lemma.
Finally, observe
\[
\langle\nu_t, 1 \rangle = \lim_{q \to\infty} \varphi( t,
q ) = \lim_{q \to\infty} \Phi\bigl( t, \varphi_0(q)
\bigr) = \Phi\bigl(t, \langle\nu_0, 1 \rangle\bigr).
\]
Thus if $\nu_0 \in\mathcal{M}_F$, equation~\eqref{totalnumber} follows
from (\ref{IVPphi}) by taking $q = \langle\nu_0, 1 \rangle$.
Since~\eqref{totalnumber} implies $t \mapsto\langle
\nu_t, 1 \rangle$ is continuous on $[0,\infty)$, we conclude that
$\nu
\colon[0,\infty) \to\mathcal{M}_F$ is weakly continuous; see, for
instance, \cite{Bauer}, Theorem~30.8.
\end{pf}
%

\begin{corollary}\label{coreu}
Assume $\Psi\colon\rplus\to\rplus$ is a critical branching
mechanism, and let $\nu_0 \in\mathcal{M}_F$. Then, there exists a
unique weak solution of~\eqref{e:gsmain} with initial data $\nu_0$.
\end{corollary}

\begin{pf}
First, by Lemma~\ref{nulemma}, there exists a weakly continuous map
$\nu
\colon[0,\infty) \to\mathcal{M}_F$ such that $\varphi$, defined
by~\eqref{defnu}, satisfies~\eqref{e:back1} with initial data~\eqref
{dataphi0}. In particular, $\nu_t \in\mathcal{M}_F$ converges weakly
to $\nu_0$ as $t \to0$. By Theorem~\ref{existgs1}, $\nu$ restricted to
$\rplus$ verifies~\eqref{weakform}. Hence, by definition, $\nu$ is weak
solution of~\eqref{e:gsmain} with initial data $\nu_0$. Uniqueness of
the solution follows from Lemma~\ref{nulemma}.
\end{pf}

\section{Self-similarity and generalized Mittag--Leffler
functions}\label{s:gml}

Recall from \cite{MP04}, that the classical Smoluchowski equation,
which corresponds here to the special branching mechanism $\Psi(u) =
u^2$, admits a one-parameter family of self-similar solutions of the form
\[
\nu_t(dx) = t^{-1}F_{\rho}\bigl(t^{-{1}/{\rho}}
\,dx\bigr), \qquad t>0, 0 < \rho\leq1,
\]
where
$F_\rho$ is given by the classical Mittag--Leffler distribution function,
satisfying
%
\begin{equation}
\label{gml1} F_{\rho}(x) = \sum_{k=1}^{\infty}
\frac{(-1)^{k+1} x^{\rho k}}{\Gamma(\rho k + 1) },\qquad \int_{\rplus}\bigl(1-e^{-qx}\bigr)
F_{\rho}(dx) = \frac{1}{1 + q^{-\rho}}.
\end{equation}
%

We now discuss the existence of self-similar solutions for homogeneous
bran\-ch\-ing mechanisms of the form
%
\begin{equation}
\label{powerlawbm} \Psi(u) = \beta u^\gamma,\qquad 1 < \gamma\leq2, \beta>0.
\end{equation}
%
As in \cite{MP04}, we look for self-similar solutions of the form
%
\begin{equation}
\label{ansatz0} \nu_t(dx) = \alpha(t) F\bigl(\lambda(t)^{-1}\,dx
\bigr),
\end{equation}
where $F$ is a probability distribution and $\alpha, \lambda>0$ are
differentiable. In this case, the function $\varphi$, defined
by~\eqref
{defnu}, takes the form
%
\begin{equation}
\label{ansatz1} \varphi(t, q) = \alpha(t) \bar\varphi\bigl(q \lambda(t)\bigr),\qquad
\bar\varphi(s) = \int_{\rplus}\bigl(1-e^{-sx}
\bigr)F(dx).
\end{equation}
%
Furthermore, by Theorem~\ref{existgs1}, $\varphi$ satisfies the equation
%
\begin{equation}
\label{phigamma} \partial_t \varphi(t, q) = -\beta\varphi(t,q)^{\gamma}
\end{equation}
for all $0\leq q \leq\infty$, where
$\varphi(t,\infty) \defeq\int_{\rplus} \nu_t(dx)$.
By~\eqref{ansatz1}, $\varphi(t,\infty) = \alpha(t)$. Hence, up to the
normalization $\alpha(0^+) = \infty$, (\ref{phigamma}) gives
%
\begin{equation}
\label{ansatz2} \alpha(t) = \bigl[(\gamma-1)\beta t\bigr]^{{1}/{(1-\gamma)}}.
\end{equation}

Now, given (\ref{ansatz1}), we rewrite (\ref{phigamma}) as
\[
\frac{\alpha^{\prime}(t)}{\alpha(t)^{\gamma}} \bar\varphi \bigl(q\lambda(t)\bigr) + q \alpha(t)^{1-\gamma}
\lambda^{\prime}(t) \bar\varphi^{\prime
}\bigl(q\lambda(t)\bigr) = -
\beta\bar\varphi\bigl(q\lambda(t)\bigr)^{\gamma}.
\]
%
In terms of the variable $s \defeq q\lambda(t)$,
separation of variables yields
\[
\frac{(\gamma-1)t \lambda^{\prime}(t)}{\lambda(t)} = \frac{\bar\varphi(s) - \bar\varphi(s)^{\gamma}}{s
\bar\varphi^{\prime}(s)} = \frac{1}\rho,
\]
where we label the separation constant as $1/\rho$ for convenience.
The constant $\beta$ disappears thanks to~\eqref{ansatz2}.
%
Solving for the general solution in each case, we obtain
%
\begin{equation}
\label{gensoln} \lambda(t) = c_1 t^{{1}/{(\rho(\gamma- 1))}},
\qquad \bar\varphi(s) =
\biggl[\frac{1}{1 + c_2 s^{-\rho(\gamma-1)}} \biggr]^{
{1}/{(\gamma-1)}},
\end{equation}
where $c_1, c_2 >0$ are arbitrary constants. Taking into
account~\eqref{ansatz1}, we have $\bar\varphi(0^+) = 0$. Therefore,
$\rho> 0$. Furthermore, the fact that $\bar\varphi$ is a Bernstein
function implies that $0 < \rho\leq1$, otherwise $\bar\varphi
^{\prime\prime}$ takes positive values near $s = 0$. We obtain the
following proposition.


\begin{proposition}\label{ssform} Assume $\Psi$ is given
by~\eqref{powerlawbm}. Then \eqref{e:gsmain} admits a one-parameter
family of self-similar solutions, indexed by $\rho\in(0,1]$, of the form
%
\begin{equation}
\label{ssfamily} \mu^{\beta, \gamma, \rho}_t(dx) = \alpha(t) F_{\gamma, \rho}
\bigl(\alpha(t)^{{1}/{\rho}}\,dx\bigr),\qquad
 \alpha(t) = \bigl[(\gamma-1)\beta t
\bigr]^{{1}/{(1-\gamma)}},
\end{equation}
where
$F_{\gamma, \rho}$ is a probability measure determined by the relation
%
\begin{equation}
\label{dlaplace2} \int_{\rplus}\bigl(1-e^{-qx}\bigr)
F_{\gamma, \rho}(dx) = \biggl[\frac{1}{1 +
q^{-\rho(\gamma- 1)}} \biggr]^{{1}/{(\gamma-1)}}.
\end{equation}
More precisely, the function
%
\begin{equation}
\label{ssfamily2} \varphi^{\beta, \gamma, \rho}(t,q) = \int_{\rplus}
\bigl(1-e^{-qx}\bigr)\mu ^{\beta, \gamma, \rho}_t(dx)
\end{equation}
solves~\eqref{phigamma} with initial data $\varphi^{\beta, \gamma,
\rho}(0,q) = q^{\rho}$. In particular, $\mu^{\beta, \gamma, 1}_t$
is the
fundamental solution of~\eqref{e:gsmain}. That is, $\mu^{\beta,
\gamma,
1}_t$ is the L\'{e}vy measure for $Z_{\beta, \gamma}(t,x)$, the
continuous-state branching process with branching mechanism \eqref{powerlawbm}.
\end{proposition}

\begin{pf}
We set $c_1 = [(\gamma-1)\beta]^{{1}/{(\rho(\gamma-1))}}$
and $c_2 = 1$ in~\eqref{gensoln},
so that \eqref{ansatz1} takes the form
%
\begin{equation}
\label{mlfun} \varphi^{\beta, \gamma, \rho}(t,q) 
= \biggl[\frac{1}{(\gamma-1)\beta t + q^{-\rho(\gamma- 1)}}
\biggr]^{{1}/{(\gamma-1)}}.
\end{equation}
By construction, this function solves~\eqref{phigamma} and has initial data
$\varphi^{\beta, \gamma, \rho}(0, q) = q^{\rho}$, which is a
Bernstein function.
Hence, by Proposition~\ref{cmdphi}, formula~\eqref{fsoln}
and Theorem~\ref{existgs1},
$\varphi^{\beta,\gamma,\rho}$ has the representation in \eqref
{ssfamily2} where
$\mu^{\beta, \gamma, \rho}_t$ solves \eqref{e:gsmain}.
The remaining statements regarding the case $\rho= 1$ follow easily
from definitions; see Sections~\ref{s:CSBP} and~\ref{s:cmd}.
\end{pf}

Finally, we show that the distribution function for $F_{\gamma, \rho}$
has a generalized Mittag--Leffler structure analogous to (\ref{gml1}).

\begin{lemma}\label{gmlemma}
Suppose $F$ is a probability measure
on $\rplus$ such that for some fixed
$r$, $s>0$,
%
\begin{equation}
\label{dslaplace} \int_{\rplus} \bigl(1-e^{-qx}\bigr)
F(dx) = \biggl[\frac{1}{1 + q^{-s}} \biggr]^r,
\end{equation}
for all $q>0$. Then the distribution function of $F$ takes the form
%
\begin{equation}
\label{gml} F(x) = \sum_{k=1}^{\infty}
\frac{(r)_k}{k!} \cdot\frac{(-1)^{k+1}
x^{sk}}{\Gamma(sk + 1)},
\end{equation}
where $(r)_k$ denotes the Pochhammer symbol, or ``rising factorial'' function
\[
(r)_k = r(r+1) (r+2)\cdots(r + k - 1).
\]
\end{lemma}

\begin{remark}
A study of generalized Mittag--Leffler distribution functions of the
form (\ref{gml}) is given by Prabhakar~\cite{Prabhakar}. In particular,
we may define, as in~\cite{Prabhakar}, the family of generalized
Mittag--Leffler functions
%
\begin{equation}
\label{mittag} E^{\rho}_{\alpha, \beta}(x) = \sum
_{k=0}^{\infty} \frac{(\rho)_k
 x^k}{ k! \Gamma(\alpha k + \beta)}, \qquad \alpha, \beta,
\rho>0,
\end{equation}
in which case~\eqref{gml} has the particular form
\[
F(x) = 1 - E^r_{s, 1}\bigl(-x^s\bigr).
\]
\end{remark}

\begin{pf}
By series expansion of $(1-x)^{-r}$ at $x=0$, one easily computes that
\[
\biggl[\frac{1}{1 + q^{-s}} \biggr]^r = \sum
_{k=0}^{\infty} \frac{r(r+1)
\cdots(r+k-1)}{k!}\bigl(-q^{-s}
\bigr)^k
\]
for $|q| > 1$. Next, note that
\[
q^{-sk} = \frac{sk}{\Gamma(sk+1)} \int_0^\infty
e^{-qx} x^{sk-1} \,dx,
\]
for $k \geq1$. Since
\[
1-\int_0^{\infty} e^{-qx} F(dx) = \sum
_{k=0}^{\infty} \frac{(r)_k}{k!}
(-1)^k q^{-sk} = 1 - \sum_{k=1}^{\infty}
\frac{(r)_k}{k!} (-1)^{k+1} q^{-sk},
\]
we conclude, formally, that
%
\begin{equation}
\label{measure1} F(dx) = \sum_{k=1}^{\infty}
\frac{(r)_k}{k!} \cdot\frac{(-1)^{k+1}
sk}{\Gamma(sk+1)} x^{sk-1} \,dx.
\end{equation}
Indeed, the previous series converges for all $x>0$, has a
(probability) distribution function given by $(\ref{gml})$, and
satisfies (\ref{dslaplace}) for all $\operatorname{Re}(q) > 1$. It
follows by the identity theorem, that (\ref{dslaplace}) holds for all
$q>0$, since both the left and right-hand sides of (\ref{dslaplace})
are analytic for $\operatorname{Re}(q) >0$.
\end{pf}

\section{Scaling limits with regularly varying \texorpdfstring{$\Psi$}{Psi}}\label{s:main1}

Proposition~\ref{ssform} establishes the existence of a family of
self-similar solutions of~\eqref{e:gsmain} with power-law branching
mechanisms $\Psi(u) = \beta u^\gamma$, $1 < \gamma\leq2$. These
solutions have a scaling invariance given by
\[
\mu^{\beta, \gamma, \rho}_t(dx) = s^{{1}/{(\gamma-1)}} \mu ^{\beta,
\gamma, \rho}_{st}
\bigl(s^{{1}/{(\rho(\gamma-1))}}\,dx\bigr),
\]
for any $s > 0$. While self-similarity arises in this case due to
homogeneity of $\Psi$, we will show that branching mechanisms with an
asymptotic power-law structure admit solutions which are asymptotically
self-similar.

Recall that a function $f > 0$ is said to be \textit{regularly
varying} at
zero (resp., infinity) with index $\rho\in\mathbb{R}$ if
\[
\frac{f(tx)}{f(t)} \to x^{\rho}
\]
as $t \to0$ (resp., $t \to\infty$) for all $x >0$. If $\rho= 0$,
then $f$ is said to be \textit{slowly varying}.

\begin{theorem}\label{main1}
Let $\nu\colon\rplus\to\mathcal{M}_F$ be a weak solution of
equation~\eqref{e:gsmain}
where $\Psi$ is a critical branching mechanism which is regularly
varying at zero with index $\gamma\in(1,2]$.
\begin{longlist}[(ii)]
\item[(i)]
Suppose there exists a nonzero $\hat\nu\in\mathcal{M}_F$ and
functions $\alpha, \lambda> 0$ such that 
%
\begin{equation}
\label{limhatnu} \alpha(t)\nu_t\bigl(\lambda(t)^{-1} \,dx\bigr)
\mathop{\xrightarrow}^{w} \hat\nu(dx)\qquad \mbox{as $t \to\infty$. }
\end{equation}
Then, there exists $\rho\in(0, 1]$ such that for all $t > 0$,
%
\begin{equation}
\label{tails} \int_0^x y
\nu_t(dy) \sim x^{1-\rho} L(t, x)\qquad \mbox{as } x \to \infty,
\end{equation}
where $L(t, \cdot)$ is slowly varying at infinity. Furthermore, there
exists $c_\lambda>0$, given by~\eqref{scalelambda}, such that
%
\begin{equation}
\label{hatnurep} \hat\nu(dx) = \langle\hat\nu, 1 \rangle F_{\gamma, \rho
}\bigl(
\langle \hat\nu, 1 \rangle^{{1}/{\rho}} c_\lambda^{-1} \,dx
\bigr).
\end{equation}
Here $F_{\gamma,\rho}$ is the generalized Mittag--Leffler distribution
defined by~\eqref{dlaplace2}. Moreover, for all $t > 0$,
%
\begin{equation}
\label{e:selfsim} \alpha(s) \nu_{st}\bigl(\lambda(s)^{-1}\,dx
\bigr) \mathop{\xrightarrow}^{w} t^{
{1}/{(1-\gamma)}} \hat\nu\bigl(t^{{1}/{(\rho(1-\gamma))}}\,dx\bigr)
\qquad\mbox{as $s \to\infty$, }
\end{equation}
and the limit in~\eqref{e:selfsim} is a self-similar solution
of~\eqref
{e:gsmain} with branching mechanism 
%
\begin{equation}
\label{branchlim} \hat\Psi(u) = \beta u^\gamma,\qquad \beta= \frac{\langle\hat\nu, 1 \rangle^{1-\gamma}}{\gamma-1}.
\end{equation}

\item[(ii)] Conversely, suppose there exist $t_0 >0$, $\rho\in(0,
1]$, and $L$ slowly varying at infinity such that \eqref{tails} holds
for $t=t_0$.
Then, there exists a function $\lambda(t) \to0$, implicitly defined
by~\eqref{deflambda}, such that~\eqref{limhatnu} holds with $\alpha(t)
= \langle\nu_t, 1 \rangle^{-1}$ and $\hat\nu= F_{\gamma,\rho}$.
\end{longlist}
\end{theorem}

\begin{remark}
Note that if~\eqref{limhatmu} holds for a weak solution $\nu\colon
[0,\infty) \to\mathcal{M}_F$ with initial data $\nu_0$, then~\eqref
{tails} holds for $t = 0$. Similarly, if~\eqref{tails} holds for $t =
t_0= 0$, then the converse result holds; cf. \cite{MP04}. Indeed, the
proof below extends easily to these cases.
\end{remark}


We begin our analysis leading to the proof of Theorem~\ref{main1}
with the following pair of useful lemmas.

\begin{lemma}[(Uniform convergence lemma)]\label{uclemma}
Assume $f > 0$ is monotone and regularly varying at $x = \infty$ with
index $\rho\ne0$. Assume $h>0$. Then, for any $0 \leq\lambda\leq
\infty$,
%
\begin{equation}
\label{ucl} \frac{f(xh(x))}{f(x)} \to\lambda^{\rho}
\end{equation}
as $x \to\infty$ if and only if $h(x) \to\lambda$ as $x \to\infty$.
\end{lemma}

\begin{pf}
The result essentially follows from the uniform convergence theorem of
Karamata (see, e.g., \cite{Bingham}, Theorem~1.5.2). In
particular, if $f$ satisfies the hypotheses above, then the convergence
$f(\lambda x)/f(x) \to\lambda^{\rho}$ as $x \to\infty$ is uniform in
$\lambda$ on compact subsets of $\rplus$. Therefore, if $h(x) \to
\lambda$ as $x \to\infty$ for $0 < \lambda< \infty$, then (\ref{ucl})
holds. The cases $\lambda= 0$ and $\lambda= \infty$ then follow from
the monotonicity of $f$.

Conversely, suppose (\ref{ucl}) holds for some $0 \leq\lambda\leq
\infty$. Then, if $h(x) \nrightarrow\lambda$, there exists a
subsequence $x_n \to\infty$ such that $h(x_n) \to\nu$ for some
$0\leq
\nu\leq\infty$ with $\nu\ne\lambda$. We deduce that $f(x_n
h(x_n))/f(x_n) \to\nu^{\rho}$, which contradicts (\ref{ucl}). This
completes the proof.
\end{pf}

\begin{lemma}\label{rvlem1}
Assume $\Psi\colon\rplus\to\rplus$ is continuous and regularly
varying at $u = 0$ with index $\gamma> 1$. Further, assume $u\dvtx
\rplus
\to\rplus$ solves the ordinary differential equation
%
\begin{equation}
\label{odepsi} u^{\prime} = -\Psi(u).
\end{equation}
Then $u$ is regularly varying at infinity with index $(1-\gamma)^{-1}$.
\end{lemma}

\begin{pf}
First assume $u$ is invertible. Then it suffices to show that the
function $u^{-1}\dvtx(0, u(0^+)) \to\rplus$ is regularly varying at
$s =
u(\infty) = 0$ with index $1/ (1-\gamma)^{-1} = 1-\gamma$. In that
case, we apply Lemma~\ref{uclemma}
to the identity
\[
\bigl(x^{{1}/{(1-\gamma)}} \bigr)^{1-\gamma} = \frac
{u^{-1}(u(tx))}{u^{-1}(u(t))} =
\frac{u^{-1} ({u(tx)}/{(u(t))}
u(t) )}{u^{-1}(u(t))},
\]
to obtain
\[
\lim_{t \to\infty} \frac{u(tx)}{u(t)} = x^{{1}/{(1-\gamma)}},
\]
for all $x > 0$. Hence, $u$ is regularly varying at infinity with index
$(1-\gamma)^{-1}$.

Obviously, $u$ is decreasing when $u >0$. Therefore, to show that $u$
is invertible, we must show that $u$ does not vanish in finite time.
Writing (\ref{odepsi}) in integral form, we have
%
\begin{equation}
\label{intode} t - t_0 = \int_{u(t)}^{u(t_0)}
\frac{1}{\Psi(w)} \,dw.
\end{equation}
Thus, $u$ vanishes in finite time if and only if $\int_0^1 \frac
{1}{\Psi
(w)} \,dw < \infty$. Note that
$\Psi(s) = s^\gamma L(s)$ where $L$ is slowly varying. Also,
%
\begin{equation}\qquad
\label{vanish} \int_s^1 \frac{1}{\Psi(w)} \,dw =
\frac{s}{\Psi(s)} \int_1^{{1}/{s}} \frac{\Psi(s)}{\Psi(sw)}
\,dw = \frac{1}{s^{\gamma-1}L(s)} \int_1^{{1}/{s}}
\frac{\Psi(s)}{\Psi(sw)} \,dw.
\end{equation}
Since $\Psi$ is regularly varying at zero, the integral term on the
right-hand side is bounded away from zero for $s$ sufficiently small.
Also, $s^r L(s) \to0$ as $s \to0$ for all $r > 0$; see, for instance,
\cite{Feller}, Lemma~VIII.8.2. Hence the left-hand side of (\ref{vanish})
diverges as $s \to0$, and we conclude that $u$ is invertible.

It remains to show that $u^{-1}$ is regularly varying at zero with
index $1-\gamma$. We consider any fixed
$0 < s_0 < u(0^+)$. By a change of variables, (\ref{intode}) implies
\[
u^{-1}(s) = u^{-1}(s_0) - \int
_{s_0}^{s} \frac{1}{\Psi(w)} \,dw
\]
for all $0 < s < u(0^+)$. Therefore, by L'H\^{o}pital's rule, we obtain
for any $s>0$
\[
\lim_{\tau\to0} \frac{u^{-1}(\tau s)}{u^{-1}(\tau)} = \lim_{\tau
\to
0}
\frac{s  ( u^{-1}  )^{\prime}(\tau s)}{  ( u^{-1}
)^{\prime} (\tau)} = \lim_{\tau\to0} \frac{ {-s
}/{(\Psi(\tau s))}}{ {-1}/{(\Psi(\tau))}} = \lim
_{\tau\to
0} \frac{s \Psi(\tau)}{\Psi(\tau s)} = s^{1-\gamma}.
\]
This completes the proof.
\end{pf}

Solutions of the autonomous equation (\ref{odepsi}) have a translation
invariance which plays an important role in our analysis.
Specifically, if $\Psi> 0$ is continuous and $\Psi(0^+)= 0$ (for
instance, any critical branching mechanism), and if $u \geq v >0$ are
solutions of (\ref{odepsi}) defined on $\rplus$, then
%
\[
v(t) = u\bigl(t - \tau+ u^{-1}\bigl(v(\tau)\bigr)\bigr)
\]
for all $t, \tau> 0$.
Recall that if $\nu\dvtx E \to\mathcal{M}_F$ is a weak solution
of~\eqref
{e:gsmain},
then the function $\varphi(\cdot,q)$, defined by (\ref{defnu}), solves
(\ref{odepsi}) for all $0 \leq q \leq\infty$. In particular, the function
%
\begin{equation}
\label{etadef} \eta(t) \defeq\langle\nu_t, 1 \rangle= \varphi(t,
\infty)
\end{equation}
solves (\ref{odepsi}). Since $\varphi(t, \infty) \geq\varphi(t,q)
> 0$, we obtain the identity
%
\begin{equation}
\label{idntyphi} \varphi(t,q) = \eta\bigl(t-\tau+ {\eta}^{-1}\bigl(
\varphi(\tau,q)\bigr)\bigr)
\end{equation}
for all $q > 0$. Thanks to this identity, the characterization of
scaling limits in the case of regularly varying branching mechanisms is
relatively straightforward.

\begin{pf*}{Proof of Theorem~\ref{main1}}
Let $\varphi$ and $\eta$ be defined by~\eqref{defnu} and~\eqref
{etadef}, respectively. Assuming~\eqref{limhatnu} holds, we have
\[
\alpha(t)\eta(t) = \bigl\langle\alpha(t)\nu_t\bigl(
\lambda(t)^{-1} \,dx\bigr), 1 \bigr\rangle\to \langle\hat\nu, 1 \rangle.
\]
Moreover, taking into account~\eqref{idntyphi}, we have
%
\begin{eqnarray}
\label{ehatnu} \bigl\langle\hat\nu, 1 - e^{-qx} \bigr\rangle&=& \lim
_{t \to\infty} \bigl\langle\alpha(t)\nu_t\bigl(
\lambda(t)^{-1} \,dx\bigr), 1 - e^{-qx} \bigr\rangle
\nonumber
\\
&=& \lim_{t \to\infty} \alpha(t) \varphi\bigl(t, q\lambda(t)
\bigr)
\nonumber
\\[-8pt]
\\[-8pt]
\nonumber
&=& \lim_{t \to\infty} \alpha(t) \eta(t) \cdot\lim
_{t \to\infty
} \frac
{\eta(t[1 - {\tau}/{t} + ({1}/{t}){\eta}^{-1}
(\varphi(\tau,q\lambda(t)))])}{\eta(t)}
\\
&=& \langle\hat\nu, 1 \rangle \biggl( \lim_{t \to\infty}
\biggl[ 1 + \frac{1}{t}{\eta}^{-1}\bigl(\varphi\bigl(\tau,q
\lambda(t)\bigr)\bigr) \biggr] \biggr)^{{1}/{(1-\gamma)}},\nonumber
\end{eqnarray}
where the last equality follows from Lemmas \ref{uclemma} and \ref{rvlem1}.
Since the left-hand side is finite and independent of $\tau$, we
conclude that there exists $\chi(q) < \infty$ such that
%
\begin{equation}
\label{eqnChi} \frac{1}{t}{\eta}^{-1}\bigl(\varphi\bigl(\tau,q
\lambda(t)\bigr)\bigr) \to\chi(q),
\end{equation}
for all $\tau> 0$.
Since $\Psi>0$ and $\Psi(0^+)= 0$, the function
$\eta$, which solves~\eqref{odepsi}, is decreasing and $\eta(\infty
) =
0$. Also, by the analysis of Section~\ref{s:cmd}, $\varphi(\tau,
\cdot
)$ is increasing with $\varphi(\tau, 0) = 0$ for all $\tau> 0$. Hence,
${\eta}^{-1} (\varphi(\tau, \cdot))$ is decreasing with ${\eta}^{-1}
(\varphi(\tau, 0)) = +\infty$. Since the limit in (\ref{ehatnu}) is
nonconstant in q, we must have $\chi(q) > 0$ and $\lambda(t) \to0$
(otherwise, $\chi$ vanishes on an unbounded interval). Therefore,
\[
\frac{({1}/{t}){\eta}^{-1}(\varphi(\tau,q\lambda(t)))}{
({1}/{t}){\eta}^{-1}(\varphi(\tau,\lambda(t)))} \to\frac{\chi
(q)}{\chi
(1)} >0
\]
as $t \to\infty$. A standard rigidity lemma (see, e.g., \cite{Feller},
Lemma~VIII.8.2) implies $\chi(q) = \chi(1) q^{\hat{\rho}}$ for some
$\hat{\rho}$, and implies ${\eta}^{-1} (\varphi(\tau, \cdot))$ is
regularly varying at $q=0$ with index $\hat{\rho}$.
Note by~\eqref{eqnChi} that $\chi$ is decreasing because $\varphi(t,
\cdot)$ is decreasing and $\eta^{-1}$ is increasing.
Further $\chi$ is not constant and so $\hat{\rho} <0$.
Also, since ${\eta}^{-1}$ is regularly varying at $q=0$ with index
$1-\gamma$ (see the proof of Lemma~\ref{rvlem1}), it follows that
$\varphi(\tau, \cdot)$ is regularly varying at $q=0$ with index
$\rho=
\hat{\rho} / (1-\gamma)$ for all $\tau> 0$. Therefore,
%
\begin{equation}
\label{e:phitilde} \hat{\varphi}(q) \defeq \bigl\langle\hat{\nu},
1 -
e^{-qx} \bigr\rangle= \langle\hat\nu, 1 \rangle \bigl[1 + \chi(1)
q^{{\rho}(1-\gamma)} \bigr]^{{1}/{(1-\gamma)}}.
\end{equation}
As a pointwise limit of Bernstein functions, $\hat{\varphi}$ is a
Bernstein function. Hence, we must have $0 < \rho\leq1$, otherwise
$\hat\varphi^{\prime\prime}$ takes positive values near $q = 0$.

Now let us show that~\eqref{tails} holds. For $t > 0$, we write
\[
\varphi(t,q) = \int_0^{\infty}
\bigl(1-e^{-qx}\bigr) \nu_t (dx) \sim q^{\rho} L
\bigl(t,q^{-1}\bigr) \qquad\mbox{as } q \to0,
\]
where $L(t,\cdot)$ is slowly varying at infinity. Next, we claim that
$q \partial_q \varphi(t,q) \sim\rho\varphi(t,q)$ as $q \to0$ for all
$t > 0$. Indeed, since $\del_q^2 \varphi\leq0$, we have
\[
\frac{q {\partial_q \varphi} (t,q)}{\varphi(t,q)} \geq\frac{
{\varphi(t,xq)}/{(\varphi(t,q))} - 1}{x-1}
\]
for all $x > 1$. Hence,
\[
\liminf_{q \to0} \frac{q {\partial_q \varphi} (t,q)}{\varphi(t,q)} \geq\liminf
_{q \to0} \frac{{\varphi(t,xq)}/{(\varphi(t,q))} -
1}{x-1} = \frac{x^\rho- 1}{x-1}.
\]
Also, the reverse inequality holds if we consider $x<1$ and take the
limit supremum instead. Thus, as $x \to1$, we recover the limit $q
\partial_q \varphi(t,q)/ \varphi(t,q) \to\rho$.

Therefore, we have
\[
\partial_q \varphi(t,q) = \int_0^{\infty}
e^{-qx} x\nu_t (dx) \sim \rho q^{\rho-1} L
\bigl(t,q^{-1}\bigr) \qquad\mbox{as } q \to0.
\]
This establishes a regular variation condition on the Laplace transform
of the measure $x\nu_t(dx)$.
By a classical Tauberian result (see, e.g., \cite{MP04}, Theorem~3.2)
we obtain the following equivalent condition on the distribution function:
\[
\int_0^x y \nu_t(dy) \sim
x^{1-\rho} L(t,x) \cdot\frac{\rho
}{\Gamma
(2-\rho)} \qquad\mbox{as } x \to\infty.
\]
Hence, redefining $L$ by a multiplicative factor, we obtain~\eqref{tails}.

Finally, let us verify~\eqref{e:selfsim}. A slight variation of
estimate~\eqref{ehatnu} gives, for all $0 \leq q \leq\infty$,
%
\begin{eqnarray}
\label{l:weaklim} 
&&\lim_{s \to\infty} \bigl\langle\alpha(s)
\nu_{st}\bigl(\lambda(s)^{-1} \,dx\bigr),  1 -
e^{-qx} \bigr\rangle
\nonumber
\\
&&\qquad = \lim_{s \to\infty} \alpha(s) \varphi\bigl(st, q\lambda(s)\bigr)
\nonumber
\\
&&\qquad= \langle\hat\nu, 1 \rangle \biggl( \lim_{s \to\infty} \biggl[ t +
\frac{1}{s}{\eta}^{-1}\bigl(\varphi\bigl(\tau, q\lambda(s)\bigr)
\bigr) \biggr] \biggr)^{{1}/{(1-\gamma)}}
\nonumber
\\[-8pt]
\\[-8pt]
\nonumber
&&\qquad= \langle\hat\nu, 1 \rangle \bigl[t + \chi(1) q^{{\rho
}(1-\gamma)}
\bigr]^{{1}/{(1-\gamma)}}
\\
&&\qquad= \bigl[\langle\hat\nu, 1 \rangle^{1-\gamma} t + \bigl(\langle \hat \nu, 1
\rangle^{{1}/{\rho}} \chi(1)^{{1}/{(\rho(1-\gamma))}} q \bigr)^{{\rho}(1-\gamma)}
\bigr]^{{1}/{(1-\gamma)}}
\nonumber
\\
&&\qquad= \bigl\langle\mu^{\beta, \gamma, \rho}_t\bigl(c_\lambda^{-1}
\,dx\bigr), 1 - e^{-qx} \bigr\rangle,\nonumber
\end{eqnarray}
where $\mu^{\beta, \gamma, \rho}$ is defined by~\eqref{ssfamily}, with
$\beta= (\gamma-1)^{-1} {\langle\hat\nu, 1 \rangle}^{1-\gamma}$ and
%
\begin{equation}
\label{scalelambda} c_\lambda= \langle\hat\nu, 1 \rangle^{{1}/{\rho}} \chi
(1)^{{1}/{(\rho(1-\gamma))}} = \biggl[  \frac{d}{dq} \bigg\rrvert _{q=0}
\hat\varphi\bigl(q^{{1}/{\rho}}\bigr) \biggr]^{{1}/{\rho}}
\end{equation}
chosen according with~\eqref{mlfun}.
The last equality, which is by no means obvious, follows from~\eqref
{e:phitilde}. In particular, when $\rho= 1$, we obtain the relation
$c_\lambda= \hat\varphi^{\prime}(0) = \langle x \hat\nu, 1
\rangle$;
cf. Theorem~\ref{funthm}.

Since~\eqref{l:weaklim} is valid for all $0 \leq q \leq\infty$ (note,
carefully, that we include $q = \infty$) the continuity theorem (see,
e.g., \cite{Feller}, Theorem~XIII.1.2) implies that $\alpha(s)\nu
_{st}(\lambda(s)^{-1} \,dx)$ converges vaguely to $\mu^{\beta, \gamma,
\rho}_t(c_\lambda^{-1} \,dx)$. Also, the case $q = \infty$ implies
convergence in total measure. We therefore obtain convergence in the
weak topology; see, for instance, \cite{Bauer}, Theorem~30.8. Hence,
taking into account~\eqref{ssfamily}, we obtain~\eqref
{hatnurep}--\eqref
{branchlim}. This completes the proof of part (i) of the theorem.
%
%

Now suppose there exists $t_0 >0$, $\rho\in(0,1]$, and $L$ slowly
varying at infinity such that (\ref{tails}) holds for $t=t_0$. Again,
by the Tauberian theorem,
we have $\partial_q \varphi(t_0,\cdot)$ regularly varying at $q=0$ with
index $\rho-1$.
The regular variation of $\varphi(t_0, \cdot)$ at $q = 0$ with index
$\rho$ then follows from the observation
\[
\frac{\varphi(t_0, q)}{q \partial_q \varphi(t_0,q)} = \frac
{\varphi
(t_0, q) - \varphi(t_0,0)}{q \partial_q \varphi(t_0,q)} = \int_0^1
\frac
{\partial_q\varphi(t_0, qz)}{\partial_q\varphi(t_0, q)} \,dz \to \frac
{1}{\rho} \qquad\mbox{as } q \to0.
\]
The convergence of the integral term is easy to verify; see, for
instance, \cite{MP08}, Lemma~3.3.

Finally, for $s>0$, let $\alpha(s) = \eta(s)^{-1}$ and define
$\lambda
(s)$ by the relation
%
\begin{equation}
\label{deflambda} \frac{1}{s}{\eta}^{-1}\bigl(\varphi
\bigl(t_0,\lambda(s)\bigr)\bigr) = 1.
\end{equation}
It follows that
\[
\frac{1}{s}{\eta}^{-1}\bigl(\varphi\bigl(t_0,q
\lambda(s)\bigr)\bigr) \to q^{\rho
(\gamma-1)},
\]
and we conclude, as above, that
\begin{eqnarray*}
\lim_{s \to\infty} \bigl\langle\eta(s)^{-1}
\nu_{st}\bigl(\lambda(s)^{-1} \,dx\bigr), 1 -
e^{-qx} \bigr\rangle&=& \bigl[t + q^{{\rho
}(1-\gamma
)} \bigr]^{{1}/{(1-\gamma)}}
\\
&= &\bigl\langle\mu^{\beta, \gamma, \rho}_t(dx), 1 - e^{-qx}
\bigr\rangle
\end{eqnarray*}
for all $t > 0$ and for all $0 \leq q \leq\infty$, where $\beta=
(\gamma-1)^{-1}$. Weak convergence of the measures follows as before.
This completes the proof.
\end{pf*}

\section{Scaling limits for fundamental solutions}\label{s:main2}

In this section we show that a necessary condition for asymptotic
self-similarity of fundamental solutions is regular variation of the
branching mechanism $\Psi$. In view of Definition~\eqref{funsoln} and
property (ii) of Theorem~\eqref{cmdphi}, we consider critical branching
mechanisms $\Psi$ for which Grey's condition holds. Our main result is
the following:

\begin{theorem}\label{funthm}
Let $\mu\colon\rplus\to\mathcal{M}_F$ be the fundamental solution
of~\eqref{e:gsmain}, where $\Psi$ is a critical branching mechanism
verifying Grey's condition.
Further, assume there exists a nonzero probability measure $\hat\mu
\in
\mathcal{M}_F$ and a function $\lambda> 0$ such that
%
\begin{equation}
\label{limhatmu} \frac{
\mu_t(\lambda(t)^{-1} \,dx)
}{\langle\mu_t, 1 \rangle} \mathop{\xrightarrow}^{w} \hat\mu
\end{equation}
as $t \to\infty$. Then
$\Psi$ is regularly varying at $u=0$ with index $\gamma\in(1,2]$.
Furthermore, $x\hat\mu\in\mathcal{M}_F$ and $\lambda(t) / \langle
\mu
_t,1 \rangle\to\langle x \hat\mu, 1 \rangle$ as $t\to\infty$.
Moreover, we have the representation
%
\begin{equation}
\label{e:selfsim2} \hat\mu= F_{\gamma, 1}\bigl(\langle x \hat\mu, 1
\rangle^{-1} \,dx\bigr),
\end{equation}
where $F_{\gamma, 1}$ is the generalized Mittag--Leffler distribution
defined by~\eqref{dlaplace2}.
\end{theorem}

Before proving Theorem~\ref{funthm}, we discuss a few basic properties
of fundamental solutions.
Let us define, as before, the total measure function
\[
\eta(t) \defeq\langle\mu_t, 1 \rangle= \Phi(t, \infty),
\]
where $\Phi$ is given by~\eqref{funphi}. Since $\Phi(t, \cdot)$ is
increasing, we have
\[
\lim_{t\to0^+} \eta(t) > \lim_{t \to0^+} \Phi(t,q)
= q
\]
for all $q>0$. Hence, $\eta(t) \to\infty$ as $t \to0$.
Moreover, $\eta$
solves~\eqref{odepsi}, where~$\Psi>0$ and $\Psi(0^+) = 0$. Hence,
$\eta(t)$ decreases to zero as $t \to\infty$. It follows that $\eta
\dvtx
\rplus\to\rplus$ is bijective, and it is straightforward to check that
its inverse is given by
%
\begin{equation}
\label{f:zeta} \zeta(\tau) \defeq\eta^{-1}(\tau) = \int
_\tau^{\infty} \frac{1}{\Psi(u)} \,du.
\end{equation}
With this notation, $\Phi$ in \eqref{funphi} has the representation
%
\begin{equation}
\label{e:etazeta} \Phi(t,q) = \eta\bigl(t + \zeta(q)\bigr),
\end{equation}
which is a special case of~\eqref{idntyphi}. From this it follows
easily that $\Phi$ satisfies the forward equation
%
\begin{equation}
\label{e:fwd} \partial_t\Phi+ \Psi(q)\partial_q\Phi=
0.
\end{equation}

Finally, we note the following useful estimates.

\begin{lemma}\label{psilemma}
Assume $\Psi$ is a critical branching mechanism that satisfies Grey's
condition,
and assume $\zeta$ is defined by~\eqref{f:zeta}.
Then the following hold for all $s>0$:
\begin{longlist}[(iii)]
\item[(i)] $ \frac{d}{ds}  [ \frac{\Psi(s)}{s^2}
] \leq0 \leq\frac{d}{ds}  [ \frac{\Psi(s)}{s}  ]$,
\item[(ii)] $ \frac{d^2}{ds^2}  [\frac{1}{\Psi
(s)} ] \geq0$,
\item[(iii)] $ \frac{d}{ds} [s \zeta(s)] \geq0$.
\end{longlist}
\end{lemma}

\begin{pf}
Part (i) is equivalent to the estimate
%
\begin{equation}
\label{cm2d} \frac{\Psi(s)}{s} \leq\Psi^{\prime}(s) \leq
\frac{2\Psi(s)}{s}.
\end{equation}
The first inequality in (\ref{cm2d}) follows from the convexity of
$\Psi
$. That is,
\[
\frac{\Psi(s)}{s} = \frac{\Psi(s)-\Psi(0)}{s-0} \leq\Psi^{\prime}(s).
\]
%
Similarly, the concavity of $\Psi^{\prime}$ gives the estimate
$\Psi^{\prime\prime}(s) \leq{\Psi^{\prime}(s)}/{s}$, which implies
\[
s\Psi^{\prime}(s) 
= \int_0^s
\bigl[\tau\Psi^{\prime\prime}(\tau) + \Psi^{\prime}(\tau)\bigr] \,d\tau\leq
\int_0^s 2\Psi^{\prime}(\tau) \,d\tau= 2
\Psi(s).
\]
For the proof of (ii), we compute
\[
\frac{d^2}{ds^2} \biggl[\frac{1}{\Psi(s)} \biggr] = \frac{2\Psi
^{\prime
}(s)^2 - \Psi^{\prime\prime}(s)\Psi(s)}{\Psi(s)^3},
\]
which is nonnegative by the estimate
\[
\frac{\Psi^{\prime\prime}(s)}{\Psi^{\prime}(s)} \leq\frac{1}{s} \leq \frac{\Psi^{\prime}(s)}{\Psi(s)} \leq
\frac{2\Psi^{\prime
}(s)}{\Psi(s)}.
\]
Finally, for the proof of (iii), observe that
%
\begin{eqnarray}
\label{estimate1} \frac{d}{ds} \bigl[s \zeta(s)\bigr] 
&=& \int_s^\infty\frac{1}{\Psi(u)} \,du -
\frac{s}{\Psi(s)} %
\nonumber
\\[-8pt]
\\[-8pt]
\nonumber
&=& \int_1^\infty
\frac{s}{\Psi(su)} \,du - \int_1^{\infty}
\frac
{s}{\Psi
(s)} \cdot\frac{1}{u^2} \,du.
\end{eqnarray}
For $u \geq1$, part (i) implies
\[
\frac{\Psi(su)}{(su)^2} \leq\frac{\Psi(s)}{s^2}.
\]
Hence (\ref{estimate1}) is nonnegative, and the proof is complete.
\end{pf}

\begin{pf*}{Proof of Theorem~\ref{funthm}}
Assuming~\eqref{limhatmu}, we have for all $0 < q \leq\infty$ that
\begin{eqnarray*}
\frac{\eta(t+\zeta(q\lambda(t)))}{\eta(t)} 
= \biggl\langle \frac{\mu_t(\lambda(t)^{-1} \,dx)}{ \langle\mu_t, 1 \rangle},
1-e^{-qx} \biggr\rangle
\to \hat\varphi(q) \defeq \bigl
\langle\hat\mu, 1-e^{-qx} \bigr\rangle
\end{eqnarray*}
as $t \to\infty$. Equivalently, in terms of the variables
%
\[
\tau\defeq\eta(t), \qquad\ell(\tau) \defeq\lambda\bigl(\zeta(\tau)\bigr),
\]
we have
%
\begin{equation}
\label{phihat} \tilde\varphi(\tau, q) \defeq\frac{\eta(\zeta(\tau)+\zeta
(q\ell(\tau
)))}{\tau} \to\hat\varphi(q)
\end{equation}
as $\tau\to0$.
Note that $\hat\varphi\colon[0,\infty) \to[0,\infty)$ is increasing
from $0$ to $1$, since $\hat\mu$ is a probability measure. Also, note
that $\tilde\varphi$ is implicitly determined by the relation
%
\begin{equation}
\label{phitilde} \zeta\bigl(q\ell(\tau)\bigr) = \zeta\bigl(\tau\tilde{\varphi}(
\tau, q)\bigr) - \zeta(\tau),
\end{equation}

\begin{claim*}
$\ell(\tau)\to0$ as $\tau\to0$.
\end{claim*}

\begin{pf}
Consider
%
\begin{equation}
F(\tau) \defeq\frac{\tau}{\Psi(\tau)}.
\end{equation}
By part (i) of Lemma~\ref{psilemma}, $F$ is nonincreasing, and
$\tau\mapsto\tau F(\tau)$ is nondecreasing, hence
%
\begin{equation}
\label{rescaledF} \qquad 1 \leq\frac{F(\tau u)}{F(\tau)} \leq\frac{1}{u}\qquad
\mbox{for $u
\leq1$},\qquad \frac{1}{u} \leq\frac{F(\tau u)}{F(\tau)} \leq1\qquad \mbox{for $u > 1$}.
\end{equation}
%
Therefore, equations (\ref{f:zeta}) and (\ref{phitilde}) imply
\[
\frac{\zeta(q\ell(\tau))}{F(\tau)} = \int_{\tilde\varphi(\tau, q)}^1
\frac{F(\tau v)}{F(\tau)} \frac{dv}{v} \geq\int_{\tilde\varphi
(\tau,
q)}^1
\frac{dv}{v}.
\]
As $\tau\to0$, the right-hand side is bounded away from zero for fixed
$q>0$. Also, $F(\tau) \to\infty$ since $\Psi(0^+) = 0$ and
$\Psi^{\prime}(0^+) = 0$. It follows that $\zeta(q\ell(\tau)) \to
\infty$ as $\tau\to0$. Since $\zeta$ is decreasing on $\rplus$ and
$\zeta(0^+) = \infty$, the claim follows.
\end{pf}

We now consider the rescaled equation
%
\begin{equation}
\label{phitilde2} \zeta_s\bigl(q\ell_s(\tau)\bigr) =
\zeta_s\bigl(\tau\tilde{\varphi}(s\tau, q)\bigr) - \zeta
_s(\tau),
\end{equation}
where
%
\begin{equation}
\zeta_s(\tau) \defeq\frac{\zeta(s\tau)}{\zeta(s)}, \qquad\ell_s(\tau)
\defeq\frac{\ell(s\tau)}{s}.
\end{equation}
We will show that as $s \to0$ a nontrivial limiting version of~\eqref
{phitilde2} holds.
That is,
%
\begin{equation}
\label{phihatlimeq} \hat{\zeta}\bigl(q \hat{\ell}(\tau)\bigr) = \hat{\zeta}\bigl(\tau
\hat{\varphi }(q)\bigr) - \hat{\zeta}(\tau)
\end{equation}
holds where
%
\begin{equation}
\label{hatlimits} \ell_s(\tau) \mathop{\xrightarrow}^{s \to0} \hat\ell(\tau) =
\hat\varphi ^{\prime
}(0)\tau,\qquad \zeta_s(\tau) \mathop{\xrightarrow}^{s \to0}
\hat\zeta(\tau) = \tau^{-r},
\end{equation}
for some $r \in(0,1]$. We will then show that the previous limits imply
that $\Psi$ is regularly varying at zero with index $\gamma= r + 1$ and
that $\hat\varphi$ has a generalized Mittag--Leffler form determined by
%
\begin{equation}
\label{hatvarphi} \hat{\varphi}(q)= \biggl[\frac{1}{
1+(\hat\varphi^{\prime}(0) q)^{-r}} \biggr]^{{1}/r}.
\end{equation}
The main idea is to show that subsequential limits of (\ref{phitilde2})
exist and are unique. We divide the proof into three main steps.

\textit{Step} 1. (\textit{Existence of subsequential limits.})

First, we write
%
\begin{equation}
\label{rescaledphi}\qquad \zeta_s(\tau) -1 = \int_1^\tau
\zeta_s^{\prime}(u) \,du 
= \int_1^{\tau}
\frac{s \zeta^{\prime}(su)}{\zeta(s)} \,du 
= \frac{-s\zeta^{\prime}(s)}{\zeta(s)} \int_{\tau}^1
\frac{\Psi(s)}{\Psi(su)} \,du.
\end{equation}
Note that for fixed $s>0$, the function
\[
\Psi_s(u) \defeq\frac{\Psi(su)}{\Psi(s)}
\]
is increasing and convex. Furthermore, by part (i) of Lemma~\ref
{psilemma}, we have
%
\begin{equation}
\label{rescaledpsi} \quad u^2 \leq\Psi_s(u) \leq u \qquad\mbox{for $u
\leq1$},\qquad u \leq\Psi_s(u) \leq u^2 \qquad\mbox{for $u > 1$}.
\end{equation}
On the other hand, by part (iii) of Lemma~\ref{psilemma},
for all $s >0$ we have
%
\begin{equation}
\label{beta} \xi(s) \defeq\frac{-s\zeta^{\prime}(s)}{\zeta(s)} \in(0,1].
\end{equation}

\begin{claim*}
$ \limsup_{s \to0} \xi(s) >0$.
\end{claim*}

\begin{pf}
Assume for the sake of contradiction that $\xi(s) \to0$ as $s \to0$.
Then, by (\ref{rescaledphi}) and (\ref{rescaledpsi}), $\zeta$ is slowly
varying at $u = 0$.
By Helly's selection theorem, there exists a sequence $\tau_j \to0$
and a function $1 \leq f(u) \leq1/u$ such that
\[
\frac{F(\tau_j u)}{F(\tau_j)} \to f(u)
\]
pointwise for $u\in(0,1)$. Since $\zeta$ is slowly varying and $\ell
(\tau_j) \to0$, we have
\[
\int_{\hat\varphi(q)}^1 \frac{f(u)}{u} \,du = \lim
_{j \to\infty} \frac
{\zeta(q\ell(\tau_j))}{F(\tau_j)} = \lim_{j \to\infty}
\frac
{\zeta
(q\ell(\tau_j))}{\zeta(\ell(\tau_j))} \cdot\frac{\zeta(\ell
(\tau
_j))}{F(\tau_j)} = \lim_{j \to\infty}
\frac{\zeta(\ell(\tau
_j))}{F(\tau_j)}.
\]
This gives a contradiction since $\hat\varphi$ is nonconstant and the
right-hand side is independent of $q$.
Therefore, the claim holds.
\end{pf}

Now we may apply Helly's selection theorem to find a sequence $s_k \to
0$ and a function $\hat{\Psi} > 0$ for which
%
\begin{eqnarray}
\xi(s_k) &\to&\hat{\xi}\qquad \mbox{for some } \hat{\xi} \in(0, 1],
\nonumber
\\[-8pt]
\\[-8pt]
\nonumber
\Psi_{s_k}(u) &\to&\hat{\Psi}(u)\qquad \mbox{for all } u> 0.
\end{eqnarray}
Furthermore, as a pointwise limit of convex functions, $\hat{\Psi}$
is convex.
By dominated convergence,
%
\begin{equation}
\label{Phihat} \zeta_{s_k}(\tau) \to\hat{\zeta}(\tau) \defeq1 + \hat{
\xi}\int_{\tau
}^1 \frac{1}{\hat{\Psi}(u)} \,du.
\end{equation}
Since $\hat\Psi$ is convex and positive,
$\hat{\zeta} \in C^1(\rplus)$ and is strictly decreasing,
and the convergence in \qref{Phihat} occurs locally uniformly for
$\tau\in\rplus$.
Now by assumption~(\ref{phihat}), the right-hand side of (\ref
{phitilde2}) converges to
%
\begin{equation}
\label{rhslimit} R(\tau, q) \defeq\hat{\zeta}\bigl(\tau\hat{\varphi}(q)\bigr) -
\hat{\zeta }(\tau ) >0
\end{equation}
for all $\tau, q >0$. Hence, the left-hand side of (\ref{phitilde2})
also converges, and if $\hat{\zeta}^{-1}$ is defined on $\rplus$, then
%
\begin{equation}
\label{lambdahat} \ell_{s_k}(\tau) \to\hat{\ell}(\tau) \defeq
q^{-1} \hat{\zeta }^{-1}\bigl(R(\tau, q)\bigr).
\end{equation}

\begin{claim*}
$\hat{\zeta}\colon\rplus\to\rplus$ is a
bijection.
\end{claim*}

\begin{pf}
Recall $\hat\zeta$ is strictly decreasing.
For $\tau\leq1$, estimate (\ref{rescaledpsi}) implies
\[
\hat{\zeta}(\tau) \geq1 + \hat{\xi}\int_{\tau}^1
\frac{1}{u} \,du = 1 - \hat{\xi} \ln\tau\to\infty\qquad\mbox{as $\tau\to0$}.
\]
%
It remains to
show that $\hat{\zeta}(\tau) \to0$ as $\tau\to\infty$. Assume, for
the sake of contradiction, that $\hat{\zeta}(\tau) \to L > 0$ as
$\tau
\to\infty$. Since $\hat{\varphi}(q) \to1$ as $q \to\infty$, we may
choose for any $\tau> 0$ a value $\hat q > 0$ sufficiently large so that
$R(\tau, \hat q) < L$. It follows
%
\begin{equation}
\label{lklim} \ell_{s_k}(\tau) \to\infty\qquad\mbox{as $k \to\infty$},
\end{equation}
for otherwise, along some bounded subsequence, the left-hand side of
(\ref{phitilde2}) would have a subsequential limit with value larger
than $L$, which is a contradiction.
But now (\ref{lklim}) implies that for all $q > 0$
\[
R(\tau, q) = \lim_{k \to\infty} \zeta_{s_k}\bigl(q
\ell_{s_k}(\tau)\bigr) \leq1,
\]
since $\zeta_s(u) \leq1$ for any $s >0$ and $u \geq1$. However, by
(\ref{rhslimit}), $R(\tau, q) \to\infty$ as $q \to0$, since $\hat
{\zeta}$ is unbounded above and $\hat{\varphi} \to0$ as $q \to0$.
This is a contradiction, which gives the claim.
\end{pf}

We conclude that, along the sequence $s_k \to0$, equation (\ref
{phitilde2}) has a well-defined limit of the form (\ref{phihatlimeq})
for all $\tau, q >0$. Furthermore, by (\ref{phitilde2}) and (\ref
{rescaledphi}) we have
\[
1+\xi(s_k) \int_{q \ell_{s_k}(\tau)}^1
\frac{1}{\Psi_{s_k}(u)} \,du = \xi (s_k) \int_{\tau\tilde\varphi(s_k \tau, q)}^\tau
\frac{1}{\Psi
_{s_k}(u)} \,du.
\]
In particular, fixing $q > 0$ and taking into account (\ref
{rescaledpsi}) shows that
%
\begin{equation}
\label{uniform} \ell_{s_k}(\tau) \to\hat{\ell}(\tau) \qquad\mbox{locally
uniformly for } \tau\in\rplus.
\end{equation}
This fact will play a role in the uniqueness proof to follow.

\textit{Step} 2. (\textit{Uniqueness of subsequential limits.})

We now show that subsequential limits obtained as in step 1 are unique.
First, equations (\ref{phihatlimeq}) and (\ref{Phihat}) imply
%
\begin{equation}
\label{limeq1a}\qquad \hat\zeta(\tau) = \hat\xi\int_{\tau\hat\varphi(q)}^{q
\hat\ell(\tau)}
\hat F(s) \frac{ds}{s} = \hat\xi\int_{\tau
{\hat\varphi(q)}/{q}}^{\hat\ell(\tau)}
\hat F(q s) \frac{ds}{s}, \qquad\hat{F}(u) \defeq\frac{u}{\hat{\Psi}(u)}.
\end{equation}
Further, since $\hat\varphi$ is concave, it follows that
\[
\hat\varphi^{\prime}(q) \leq\frac{\hat\varphi(q)}{q} \leq\frac
{\hat
\ell(\tau)}{\tau}\qquad
\mbox{for all } \tau, q > 0.
\]
Therefore, $\hat\varphi^{\prime}(q) = \int_E e^{-qx} x \hat\mu
(dx)$ is
decreasing and bounded above, and we deduce that
%
\[
0 < \hat\varphi^{\prime}_0 \defeq\hat\varphi^{\prime}
\bigl(0^+\bigr) < \infty.
\]
Furthermore, taking $q \to0$ in (\ref{limeq1a}) gives
%
\begin{equation}
\label{limeq1b} \hat\zeta(\tau) = \hat\xi\hat F_0 \ln
\frac{\hat\ell(\tau
)}{\hat\varphi
^{\prime}_0 \tau}\qquad\mbox{with } \hat F_0 \defeq\hat F\bigl(0^+
\bigr).
\end{equation}
%
Formally, we have $\hat F_0 = \infty$ if and only if $\hat\ell(\tau)
= \hat\varphi^{\prime}_0 \tau$ for all $\tau>0$. More precisely, note
that the left-hand side of (\ref{limeq1a}) is positive, so that $\hat
F(qs)$ has a finite limit as $q \to0$ if and only if $\hat
\ell(\tau) > \hat\varphi^{\prime}_0 \tau$ for each $\tau>0$.

On the other hand, equations (\ref{phihatlimeq}) and (\ref{Phihat})
also imply
%
\begin{equation}
\label{limeq2a} \hat\zeta\bigl(q \hat\ell(\tau)\bigr) = \hat\xi\int
_{\tau\hat\varphi
(q)}^{\tau} \hat F(s) \frac{ds}{s} = \hat\xi
\int_{\hat\varphi(q)}^{1} \hat F(\tau s) \frac{ds}{s}.
\end{equation}
Taking $\tau\to0$ implies
%
\begin{equation}
\label{limeq2b} \hat\zeta(q \hat\ell_0) = \hat\xi\hat
F_0 \ln \frac{1}{\hat\varphi(q)} \qquad\mbox{with } \hat\ell_0
\defeq\hat\ell\bigl(0^+\bigr).
\end{equation}
%
In particular, $\hat\ell_0 = 0$ if and only if $\hat F_0 = \infty$.
Note also $\hat\ell_0 < \infty$, since $\hat F_0 \geq1$.

We now consider two cases.

\textit{Case} 1: (\textit{$\hat\ell_0 = 0$}). As noted above, $\hat
\ell_0 =
0$ if and only if $\hat\ell(\tau) = \hat\varphi^{\prime}_0 \tau$ for
all $\tau>0$. Hence, (\ref{phihatlimeq}) reduces to
%
\begin{equation}
\label{phihatlimeq1} \hat{\zeta}\bigl(\tau q \hat\varphi^{\prime}_0
\bigr) = \hat{\zeta}\bigl(\tau \hat {\varphi}(q)\bigr) - \hat{\zeta}(\tau).
\end{equation}
Differentiating in $q$ and $\tau$ gives the relations
%
\begin{equation}
\label{relations} \frac{\hat{\varphi}(q)}{q \hat{\varphi}^{\prime}(q)} = \frac{\hat{F}(\tau\hat{\varphi}(q))}{\hat{F}(\tau q \hat\varphi
^{\prime
}_0)} = 1 +
\frac{\hat{F}(\tau)}{\hat{F}(\tau q \hat\varphi
^{\prime}_0)}.
\end{equation}
Therefore, $\hat{F}(\tau q \hat\varphi^{\prime}_0) / \hat F(\tau)$ is
constant in $\tau$ and we deduce that $\hat F$ is a power law:
$\hat F(u) = u^{-r}$, since $\hat F(1)=1$.
Note that $r \ne0$, since $\hat F_0 = \infty$. It then follows from~(\ref{rescaledpsi}) and (\ref{limeq1a}) that $0 <r \leq1$.
The second equality above reduces to
%
\begin{equation}
\label{powerlaw} \frac{(q \hat\varphi^{\prime}_0)^r}{\hat\varphi(q)^r} = 1 + \bigl(q \hat \varphi^{\prime}_0
\bigr)^r,
\end{equation}
which gives (\ref{hatvarphi}). On the other hand, $\hat\Psi(u) =
u^{r+1}$, so that (\ref{Phihat}) implies
\[
0 = \hat\zeta(\infty) = 1 - \hat\xi\int_1^\infty
\frac{1}{u^{r+1}} \,du = 1 - \frac{\hat\xi}{r}.
\]
Hence, $\hat\xi= r$. In summary, we obtain in this case
%
\begin{equation}
\label{hatlimits2} 
\hat\ell(\tau) = \hat\varphi^{\prime}_0
\tau,\qquad 
\hat\Psi(\tau) = \tau^{r+1},\qquad 
\hat\zeta(
\tau) = \tau^{-r},\qquad 
\hat\xi= r,
\end{equation}
where $0 < r \leq1$ and $\hat\varphi$ is given by (\ref{hatvarphi}).

\textit{Case} 2: (\textit{$\hat\ell_0 > 0$}). We will show that the
remaining case, $\hat\ell_0 > 0$, leads to a contradiction. We divide
this case into three parts.

(i) First, let us show that if $\hat\ell_0 > 0$, then $\hat
\varphi$ has the form (\ref{hatvarphi}), and
%
\begin{equation}
\label{eqhatlambda} \hat\ell(\tau)^r = \hat\ell_0^r
+\bigl(\tau\hat\varphi_0^{\prime}\bigr)^r.
\end{equation}
The idea is to consider a rescaling of (\ref{phihatlimeq}), of the same
form as (\ref{phitilde2}); namely,
%
\begin{equation}
\label{phihatlimeq2} \hat{\zeta}_s\bigl(q \hat{\ell}_s(\tau)
\bigr) = \hat{\zeta}_s\bigl(\tau\hat {\varphi }(q)\bigr) - \hat{
\zeta}_s(\tau),
\end{equation}
where
%
\begin{equation}
\hat\zeta_s(\tau) \defeq\frac{\hat\zeta(s\tau)}{\hat\zeta(s)}, \qquad\hat\ell_s(
\tau) \defeq\frac{\hat\ell(s\tau)}{s}.
\end{equation}
Since $\hat\zeta(\tau)\to0$ as $\tau\to0$, we deduce from~\eqref
{limeq1b} that
%
\begin{equation}
\label{lambdainfty} \hat\ell_s(\tau) = \tau\cdot\frac{\hat\ell(s\tau)}{s\tau} \to
\tau\hat\varphi^{\prime}_0 \qquad\mbox{as } s \to\infty.
\end{equation}
Furthermore, since Lemma~\ref{psilemma} applies to the functions $\hat
\Psi$ and $\hat\zeta$, we can use Helly's selection principle, as in
step 1, to
pass to the limit in (\ref{phihatlimeq2}) along some sequence $\hat
{s}_k \to\infty$.
Up to relabeling, the limit equation matches exactly the form (\ref
{phihatlimeq1}). In particular, (\ref{lambdainfty}) implies that $\hat
\zeta_{\hat{s}_k}$ has a nonconstant limit, and we obtain, as before,
(\ref{hatvarphi}) from the relations (\ref{relations}). Note that the
constant $r$ in (\ref{hatvarphi}) is the same as in the previous case,
since $\hat\varphi$ is fixed.

Now, substituting $q = \tau/ \hat\ell_0$ in (\ref{limeq2b}) and
comparing with (\ref{limeq1b}), we obtain
%
\begin{equation}
\frac{\hat\ell(\tau)}{\tau\hat\varphi^{\prime}_0 } = \frac
{1}{\hat
\varphi( \tau/ \hat\ell_0)}
\end{equation}
for all $\tau>0$. Using (\ref{hatvarphi}) in the previous relation
gives~\eqref{eqhatlambda}.
In particular,
%
\begin{equation}
\label{capdelta0} \frac{\hat\ell(\tau)}{\tau} = \biggl[ \frac{\hat\ell_0^r}{\tau^r} + \bigl(\hat
\varphi^{{\prime}}_0\bigr)^r \biggr]^{{1}/{r}}
\end{equation}
is decreasing as a function of $\tau> 0$.

(ii) Next, let us show that if~\eqref{capdelta0} holds, then
%
%
\begin{equation}
\label{supinflambda} \limsup_{\tau\to0} \frac{\ell(\tau)}{\tau} = \infty
\quad\mbox{and}\quad \liminf_{\tau\to0} \frac{\ell(\tau)}{\tau} = \hat
\varphi_0^{\prime}.
\end{equation}
Recall that $\hat\ell(\tau) = \lim_{k \to\infty} \ell_{s_k}(\tau)$
for some $s_k \to0$. Therefore,~\eqref{capdelta0} implies
\[
\lim_{k \to\infty} \frac{\ell(s_k \tau)}{s_k \tau} = \frac{\hat
\ell
(\tau)}{\tau} \to
\infty
\]
as $\tau\to0$, and the first statement in~\eqref{supinflambda} follows.

On the other hand, by~\eqref{phitilde}
%
\[
\frac{\ell(\tau)}{\tau} > \frac{\tilde\varphi(\tau,q)}{q}\qquad \mbox {for all} \tau, q>0.
\]
Hence, for all $t>0$ and for all $q>0$,
\[
\frac{\hat\ell(t)}{t} = \lim_{k \to\infty} \frac{\ell(s_k
t)}{s_k t} \geq
\liminf_{\tau\to0} \frac{\ell(\tau)}{\tau} > \liminf
_{\tau
\to
0} \frac{\tilde\varphi(\tau,q)}{q} = \frac{\hat\varphi(q)}{q}.
\]
Taking into account~\eqref{capdelta0} and passing to the limit $t \to
\infty$ on the left and $q \to0$ on the right yields the last
statement in~\eqref{supinflambda}.

(iii) Finally, we show that (\ref{capdelta0}) and (\ref
{supinflambda}) lead to a contradiction. Fix~$M > m > \hat\varphi
_0^{\prime}$, and choose a sequence of disjoint intervals $[a_k, b_k]$
as follows:
\begin{longlist}[(1)]
\item[(1)] Choose $b_k \to0$ such that $ \frac{\ell
(b_k)}{b_k} > M$.
\item[(2)] Define $ c_k = \sup\{\tau< b_k\dvtx\frac{\ell
(\tau)}{\tau} < m \}$.
\item[(3)] Choose $a_k$ such that $ 1 < \frac{c_k}{a_k} <
1 + \frac{1}{k}$ and $ \frac{\ell(a_k)}{a_k} < m$.
\end{longlist}
Taking $s = a_k$ and $\tau= 1$ in (\ref{phitilde2}), we have
%
\begin{equation}
\label{rescalebk} \zeta_{a_k}\bigl(q \ell_{a_k}(1)\bigr) =
\zeta_{a_k}\bigl(\tilde\varphi(a_k, q)\bigr) - 1.
\end{equation}
Since
\[
\hat\varphi_0^{\prime} \leq\liminf_{k \to\infty}
\frac{\ell
(a_k)}{a_k} \leq\limsup_{k \to\infty} \frac{\ell(a_k)}{a_k} \leq
m
\]
and
$\tilde\varphi(a_k, q) \to\hat\varphi(q)$ as $k \to\infty$, it
follows that the sequence $\xi(a_k)$, defined by (\ref{beta}), is
bounded away from zero; otherwise, there exists a subsequence $\zeta
_{a_{k_j}}(\tau) \to1$, which contradicts (\ref{rescalebk}).
Therefore, as in step 1, (\ref{phitilde2}) has a nontrivial limit along
a subsequence $a_{k_j} \to0$, $j \geq1$. In particular, the local
uniform convergence of $\ell_{a_{k_j}}$ implies
%
\begin{equation}
\label{uniform2} \frac{\ell_{a_{k_j}}(\tau)}{\tau} \to\frac{\hat\Lambda(\tau
)}{\tau} \qquad\mbox{locally
uniformly for } \tau\in\rplus,
\end{equation}
where $\hat\Lambda$ satisfies (\ref{eqhatlambda}), or, equivalently,
%
\begin{equation}
\label{capdelta} \frac{\hat\Lambda(\tau)}{\tau} = \biggl[ \frac{\hat\Lambda
_0^r}{\tau^r} + 
{
\bigl(\hat\varphi'_0\bigr)}^r
\biggr]^{{1}/{r}}.
\end{equation}
If $\hat\Lambda_0 > 0$, then~\eqref{capdelta} is strictly decreasing in
$\tau$, and we have, for all $\tau> 1$,
\[
\frac{\hat\Lambda(\tau)}{\tau} < \hat\Lambda(1) = \lim_{j \to
\infty}
\ell_{a_{k_j}}(1) \leq m.
\]
On the other hand, if $\hat\Lambda_0 = 0$, then $\hat\Lambda(\tau)
/\tau= \hat\Lambda(1) = \hat\varphi^{{\prime}}_0 < m$ for all
$\tau>
0$. Hence, in either case, we have $\hat\Lambda(\tau) /\tau< m$ is
nonincreasing for all $\tau> 1$.

Next, choose $\varepsilon< \min\{M-m, m - \frac{\hat\Lambda
(2)}{2}\}
$, and choose $J$ large enough so that
%
\begin{equation}
\label{uniform1} \biggl\llvert \frac{\ell_{a_{k_j}}(\tau)}{\tau} - \frac{\hat\Lambda
(\tau)}{\tau
} \biggr
\rrvert < \varepsilon\qquad\forall j \geq J, \forall\tau\in[1,3].
\end{equation}
Since $r_j \defeq{b_{k_j}}/{a_{k_j}} > 1$ and
\[
\biggl\llvert \frac{\ell_{a_{k_j}}(r_j)}{r_j} - \frac{\hat\Lambda(r_j)}{r_j} \biggr\rrvert = \biggl
\llvert \frac{\ell(b_{k_j})}{b_{k_j}} - \frac{\hat
\Lambda
(r_j)}{r_j} \biggr\rrvert \geq|M - m| >
\epsilon,
\]
it follows from~\eqref{uniform1} that $r_j > 3$ for all $j \geq J$. Therefore,
%
\begin{equation}
\label{bound23} \frac{\ell_{a_{k_j}}(\tau)}{\tau} \geq m \qquad\forall j \geq J, \forall
\tau\in[2, 3] \subset (c_{k_j}/a_{k_j},
r_j ],
\end{equation}
%
since, by construction, ${\ell(\tau)}/{\tau} \geq m$ for all $\tau
\in
(c_k, b_k]$. Hence,~\eqref{bound23} implies
%
\begin{equation}
\biggl\llvert \frac{\ell_{a_{k_j}}(\tau)}{\tau} - \frac{\hat\Lambda
(\tau)}{\tau
} \biggr\rrvert \geq\biggl
\llvert m - \frac{\hat\Lambda(\tau)}{\tau} \biggr\rrvert \geq \biggl\llvert m -
\frac{\hat\Lambda(2)}{2} \biggr\rrvert > \varepsilon
\end{equation}
for all $j \geq J$ and for all $\tau\in[2, 3]$. This
contradicts~\eqref{uniform1}. Therefore, the hypothesis of case 2,
$\hat\ell_0 > 0$, is never satisfied, and we obtain in step 1 unique
subsequential limits of the form~\eqref{hatlimits2}.

\textit{Step} 3. (\textit{Limit as $s \to0$.})

To finish the proof of the theorem, note that we must have $\xi(s) \to
\hat\xi= r$ as $s \to0$. Otherwise, by step 1, it is possible to
extract subsequential limits with distinct values of $\hat\xi$,
contradicting (\ref{hatlimits2}).
Similarly, the full limit of each of the rescaled functions
$\zeta_s$, $\Psi_s$, and $\ell_s$ exists as $s \to0$,
since given any sequence $s_k \to0$, there exist unique subsequential
limits by steps 1 and 2. In particular, (\ref{hatlimits2}) shows that
$\Psi$ is regularly varying with index $\gamma= r+1 \in(1,2]$. Also,
(\ref{powerlaw}) implies
\[
\bigl\langle\hat\mu, 1-e^{-qx} \bigr\rangle= \hat\varphi(q) =
\frac{1}{ [1+(\hat\varphi^{\prime}(0) q)^{-r} ]^{1/r} },
\]
where $\hat\varphi^{\prime}(0) = \langle x \hat\mu, 1 \rangle$. This
gives~\eqref{e:selfsim2}. Finally,~\eqref{hatlimits} implies that
\[
\frac{\lambda(\zeta(s\tau))}{\eta(\zeta(s\tau))}\tau= \frac
{\lambda
(\zeta(s\tau))}{s} = \ell_s(\tau) \to\hat
\varphi^{\prime
}(0)\tau.
\]
%
Hence $\lambda(t) \sim\hat\varphi^{\prime}(0) \eta(t) = \langle x
\hat
\mu, 1 \rangle$ $\langle\mu_t,1 \rangle$ as $t\to\infty$, and the
proof is complete.
\end{pf*}

\begin{remark}
The conclusions of the theorem follow much more quickly
if one \textit{assumes} that the scaling function
$\lambda(t)\sim\langle\mu_t,1 \rangle$ in \eqref{limhatmu},
based on the arguments of Pakes~\cite{Pakes} which make use
of the forward equation \eqref{e:fwd}. Testing \eqref{limhatmu}
with $xe^{-qx}$ it follows
%
\begin{equation}
\partial_q\Phi(t,\lambda q) = \frac{\Psi(\Phi(t,\lambda q))}{\Psi
(\lambda q)} \to\hat
\varphi'(q),\qquad q>0.
\end{equation}
Writing $u=\lambda q$ and noting $\theta= \hat\varphi(q)/q$ is a monotonic
function of $q$, we have that $\Phi(t,\lambda q)= u\theta(1+o(1))$
and thus
%
\begin{equation}
\frac{\Psi(u\theta(1+o(1)))}{\Psi(u)} \to h(\theta)
\end{equation}
as $u\to0$. By simple estimates based on the
continuity and monotonicity of $\Psi$, one can eliminate the $1+o(1)$
factor and conclude that $\Psi$ is regularly varying
by the standard rigidity lemma in \cite{Feller}, Lemma~VIII.8.2.
\end{remark}


\section{Limit theorems for critical CSBPs}\label{s:limCSBP}

We conclude this paper by applying the results in Sections~\ref
{s:main1} and \ref{s:main2} to derive limit theorems for critical CSBPs
that become extinct almost surely.
First, we obtain a conditional limit theorem for fixed initial
population $x$.
In particular this solves the continuous-state analog of the open
question posed by Pakes in \cite{Pakes}, Remark~6.1.

\begin{theorem}\label{limCSBP}
Assume $Z(t,x)$ is a continuous-state branching process with critical
branching mechanism $\Psi$ verifying Grey's condition. Further, assume
that for some (equivalently all) $x>0$,
there exists a function $\lambda> 0$ and a probability measure $\hat
\mu\in\mathcal{M}_F$ such that
%
\begin{equation}
\label{limcdf} \PP\bigl(\lambda(t)Z(t,x) \leq z | Z(t,x) > 0\bigr) \to\int
_{(0,z)}\hat\mu(du)
\end{equation}
%
holds for all points $z$ for which $\hat\mu(\{z\}) = 0$.
Then, there exists $1<\gamma\leq2$ such that $\Psi$ is regularly
varying at $u=0$ with index $\gamma$. Furthermore, $x \hat\mu\in
\mathcal{M}_F$ and $\lambda(t) \sim\langle x \hat\mu,1 \rangle\PP
(Z(t,1)>0)$ as $t \to\infty$.

Conversely, suppose $\Psi$ is regularly varying at $u=0$ with index $1
< \gamma\leq2$. Then,~\eqref{limcdf} holds with $\lambda(t) = \PP
(Z(t,x)>0)$ and $\hat\mu= F_{\gamma, 1}(dz)$.
\end{theorem}

\begin{pf}
It follows from~\eqref{e:laplace1} that
%
\begin{equation}
\label{probextinct} \PP\bigl(Z(t,x) = 0\bigr) 
=
\lim_{q \to\infty} \EE\bigl(e^{-qZ(t,x)}\bigr) = e^{-x\varphi(t,\infty)} =
e^{-x\langle\mu_t, 1 \rangle},
\end{equation}
%
with $\mu_t$ the L\'{e}vy measure for $Z(t,x)$.
By the continuity theorem~\cite{Feller}, Theorem~XIII.1.2, \eqref{limcdf} implies
\begin{eqnarray*}
\int_{\rplus}e^{-qy}\hat\mu(dy) 
&= &\lim_{t \to\infty}
\frac{\mathbb E( e^{-q \lambda(t) Z(t, x)}) -
\PP
(Z(t,x) = 0)}{\PP(Z(t, x) > 0)}
\\
&=& \lim_{t \to\infty} \frac{e^{-x\varphi(t,\lambda(t) q)} -
e^{-x\langle\mu_t, 1 \rangle}}{1-e^{-x\langle\mu_t, 1 \rangle}}.
\end{eqnarray*}
%
Hence,
\begin{eqnarray*}
\lim_{t\to\infty} \int_{\rplus}
\bigl(1-e^{-qy}\bigr)\frac{\mu_t(\lambda
(t)^{-1}\,dy)}{\langle\mu_t, 1 \rangle}& =& \lim_{t\to\infty}
\frac{\varphi(t,\lambda(t)q)}{\langle\mu_t, 1
\rangle}
\\
&=& \lim_{t\to\infty} \frac{1-e^{-x\varphi(t,\lambda(t)
q)}}{1-e^{-x\langle\mu_t, 1 \rangle}} = \int_{\rplus
}
\bigl(1-e^{-qy}\bigr)\hat \mu(dy),
\end{eqnarray*}
where the second equality follows by Taylor expansion and the fact that
$0<\varphi(t,\lambda(t)q)<\langle\mu(t),1\rangle\to0$ as $t\to
\infty$.
Since $\mu_t$ is the fundamental solution of the associated
equation~\eqref{e:gsmain}, we conclude, by Theorem~\ref{funthm}, that
there exists $1<\gamma\leq2$ such that $\Psi$ is regularly varying at
$u=0$ with index~$\gamma$. Also, by Theorem~\ref{funthm},
\[
\lambda(t) \sim\langle x \hat\mu,1 \rangle \langle\mu_t, 1 \rangle
\sim\langle x \hat\mu,1 \rangle\bigl(1-e^{-\langle\mu_t, 1 \rangle}\bigr) = \langle x \hat
\mu,1 \rangle\PP\bigl(Z(t,1)>0\bigr)
\]
as $t \to\infty$. The converse follows easily from Theorem~\ref
{main1}. This completes the proof.
\end{pf}

Next, based on the same results on scaling limits of fundamental solutions,
we study scaling limits as $t\to\infty$
of CSBPs with initial population scaled
to obtain nondegenerate L\'evy process limits $x\mapsto\hat Z(x)$.
As in~\cite{Jacod}, Chapter VI, let $\Skor$ denote the space of
\cadlag
paths equipped with the Skorokhod topology.
We use the notation $\laweq$ to denote equality in law (i.e., both
processes define the same measure on the Skorokhod space $\Skor$), and
the notation $\mathop{\rightarrow}\limits^ \mathcal{L}$ to denote convergence in law for these processes
(i.e., weak convergence of the induced distributions on the Skorokhod space).

For convenience, we introduce a notation for rescaled processes.
If $\lambda, \alpha> 0$, and $x\mapsto X(x)$ is a process, then we
define the rescaled process $\delta_{\lambda, \alpha}X$ by $\delta
_{\lambda, \alpha} X(x) \defeq\lambda X( \alpha x )$.

\begin{theorem}\label{ssCSBP}
Let $Z(t,x)$ be a continuous state branching process with critical
branching mechanism $\Psi$ satisfying Grey's condition.
\begin{longlist}[(i)]
\item[(i)]
Assume there exists a L\'{e}vy process $\hat Z = \hat Z (x)$ and
functions $\alpha, \lambda> 0$ such that
%
\begin{equation}
\label{processlimit} \delta_{\lambda(t), \alpha(t)} Z(t, \cdot)
\mathop{\rightarrow}^ \mathcal{L}_{t \to\infty} \hat Z(
\cdot).
\end{equation}
Further, assume the nondegeneracy condition
%
\begin{equation}
\label{ndlimit} \lim_{t \to\infty} \PP\bigl(Z\bigl(t, \alpha(t)x\bigr)
= 0\bigr) = \PP\bigl(\hat Z(x) = 0\bigr) \in(0,1) 
\end{equation}
for some $x > 0$. Then, there exists $1 < \gamma\leq2$ such that
$\Psi
$ is regularly varying at $u=0$ with index $\gamma$, and there exist
constants $c_\alpha, c_\lambda>0$ such that
%
\begin{equation}
\label{limsim} \frac{c_\alpha}{\alpha(t)} \sim\frac{\lambda(t)}{c_\lambda} \sim \PP \bigl(Z(t,1)
> 0\bigr) \qquad\mbox{as } t\to\infty.
\end{equation}
Furthermore, for all (fixed) $t > 0$
%
\begin{equation}
\label{ssprocess} 
\delta_{\lambda(s), \alpha(s)} Z(st, \cdot)
\mathop{\rightarrow}^ \mathcal{L}_{s \to\infty} \delta _{t^{\gamma^*}, t^{-\gamma^*}} \hat Z \qquad\mbox{where } \gamma^* \defeq
\frac{1}{\gamma- 1}.
\end{equation}
%
Also, for all $t > 0$,
%
\begin{equation}
\label{ssprocess1} 
\delta_{1,c_\alpha c_\lambda}
Z_{\beta, \gamma}(t, \cdot) \laweq \delta _{t^{\gamma^*}, t^{-\gamma^*}} \hat Z,
\end{equation}
where $Z_{\beta, \gamma}(t, \cdot)$ is the continuous-state branching
process with branching mechanism $\hat\Psi(u) = \beta u^{\gamma}$ with
$\beta= \gamma^*{c_\lambda^{\gamma- 1}}$.



%
\item[(ii)]
Conversely, assume $\Psi$ is regularly varying at zero with index $1 <
\gamma\leq2$. Then~\eqref{ssprocess} holds
with $\lambda(s) = \alpha(s)^{-1} = \PP(Z(s,1) > 0)$, where $\hat Z(x)$
is the L\'{e}vy process with L\'{e}vy measure $F_{\gamma, 1}(dx)$
defined by~\eqref{dlaplace2}.
\end{longlist}
\end{theorem}

\begin{pf}
Since we are dealing with increasing L\'evy processes, the process convergence
in \qref{processlimit} is equivalent to the pointwise convergence of Laplace
exponents
%
\begin{equation}
\label{scalelim1a} \alpha(t)\vp\bigl(t,\lambda(t)q\bigr)
\mathop{\xrightarrow}_{t \to\infty}
\hat \varphi (q)\qquad \mbox{for all } q \in[0, \infty),
\end{equation}
where $x \hat\varphi(q) \defeq-\ln\mathbb E(e^{-q \hat Z(x)})$ is the
Laplace exponent of $\hat Z(x)$.
(See, e.g.~\cite{Jacod}, Corollary VII.4.43 and \cite{Feller},
Theorems XV.3.2 and XIII.1.2, as in the proof'' the proof of
Theorem~1 in~\cite{MP07}.)
%
%
%
Furthermore, by~\eqref{ndlimit} we must have
\begin{eqnarray*}
\lim_{t \to\infty} \lim_{q \to\infty} x \alpha(t) \varphi
\bigl(t, \lambda (t) q\bigr) &=& \lim_{t \to\infty} \lim
_{q \to\infty} - \ln\EE\bigl(e^{-q
Z^{(t)}(x)}\bigr)
\\
&=& \lim_{t \to\infty} -\ln\PP\bigl(\lambda(t) Z\bigl(t, \alpha(t)x
\bigr) = 0\bigr)\\
& =& -\ln\PP\bigl(\hat Z(x) = 0\bigr) = \lim_{q \to\infty}
x \hat\varphi(q),
\end{eqnarray*}
and hence,
%
\begin{equation}
\label{scalelim2} \lim_{t \to\infty} \alpha(t)\varphi(t, \infty) = \hat
\varphi (\infty) \in(0, \infty).
\end{equation}
Then, denoting the L\'{e}vy measures of $Z(t,x)$ and $\hat Z(x)$ by
$\mu
_t$ and $\hat\mu$, respectively, we deduce from~\eqref{scalelim1a}
and~\eqref{scalelim2} that
%
\begin{equation}
\frac{1}{\langle\mu_t, 1 \rangle} \mu_t\bigl(\lambda(t)^{-1} \,dx\bigr)
\mathop{\xrightarrow}^{w} \frac{1}{\langle\hat\mu, 1 \rangle} \hat\mu(dx).
\end{equation}
Therefore, by Theorem~\ref{funthm}, there exists $1 < \gamma\leq2$
such that $\Psi$ is regularly varying at $u=0$ with index $\gamma$. Moreover,
\[
\lambda(t) \sim\frac{\langle x \hat\mu,1 \rangle}{\langle\hat
\mu,1
\rangle} \langle\mu_t, 1 \rangle\sim
\frac{\langle x \hat\mu,1
\rangle}{\langle\hat\mu,1 \rangle} \bigl(1-e^{-\langle\mu_t, 1
\rangle}\bigr) = \frac{\langle x \hat\mu,1 \rangle}{\langle\hat\mu,1 \rangle} \PP
\bigl(Z(t,1)>0\bigr).
\]
Hence, together with~\eqref{scalelim2}, we obtain~\eqref{limsim} with
$c_\alpha= \langle\hat\mu, 1 \rangle$ and $c_\lambda= \frac
{\langle
x \hat\mu,1 \rangle}{\langle\hat\mu,1 \rangle}$.
Also, by Theorem~\ref{funthm},
\[
\hat\mu(dx) = c_\alpha F_{\gamma, 1} \bigl(c_\lambda^{-1}
\,dx\bigr) =c_\alpha c_\lambda \bigl[c_\lambda^{-1}
F_{\gamma, 1} \bigl(c_\lambda^{-1} \,dx\bigr) \bigr]
=c_\alpha c_\lambda\mu^{\beta, \gamma, 1}_1(dx),
\]
where $\mu^{\beta, \gamma, 1}$ is defined by~\eqref{ssfamily} with
$\beta= \frac{c_\lambda^{\gamma- 1}}{\gamma- 1}$. Therefore, by
Theorem~\ref{main1}, we have
%
\begin{equation}
\label{ssprocess2} \alpha(s)\varphi\bigl(st, \lambda(s)q\bigr) \to t^{-\gamma^*}
\hat\varphi \bigl(t^{\gamma^*}q\bigr) = c_\alpha c_\lambda
\varphi^{\beta,
\gamma,
1}(t, q)
\end{equation}
as $s \to\infty$ for all $0 \leq q \leq\infty$, where $\varphi
^{\beta, \gamma, 1}$ is defined by~\eqref{ssfamily2}.
Since
\begin{eqnarray*}
\EE\bigl(e^{-q\lambda(s)Z(st, \alpha(s) x)}\bigr) &=& e^{-x\alpha(s)\varphi(st,
\lambda(s)q)}
\\
&\mathop{\xrightarrow}\limits^{s \to\infty}& e^{-x t^{-\gamma^*}
\hat\varphi(t^{\gamma^*}q) } = \EE\bigl(e^{-qt^{\gamma^*} \hat
Z(t^{-\gamma^*}x)}\bigr),
\end{eqnarray*}
we obtain~\eqref{ssprocess}. Similarly, we obtain~\eqref{ssprocess1}
from~\eqref{ssprocess2}.

For the converse, we recall that the convergence in \eqref{processlimit} holds if and only
if the Laplace exponent converges pointwise as in \eqref{scalelim1a}. The converse
then follows by a similar argument.
\end{pf}

\begin{remark} The nondegeneracy condition \eqref{ndlimit} has the following
interpretation. The spatial process $x\mapsto Z^{\alpha, \lambda
}_t(x) $
is a compound Poisson process with jump measure $\alpha(t)\mu
_t(\lambda
(t)^{-1}\,dx)$
and scaled intensity $\alpha(t)\langle\mu_t,1 \rangle$.
One thing that \eqref{ndlimit} means is that we assume the scaled
intensity converges
to the intensity of jumps $\langle\hat\mu,1 \rangle=1$ in the
limiting process
$\hat Z$.
In particular this presumes there are no small jumps with finite intensity
being lost in the limit.
\end{remark}


%



\printaddresses

\end{document}